\documentclass[a4paper,english,fontsize=10pt,parskip=half,abstracton]{scrartcl}
\usepackage{babel}
\usepackage[utf8]{inputenc}
\usepackage[T1]{fontenc}
\usepackage[a4paper,left=20mm,right=20mm,top=30mm,bottom=30mm]{geometry}
\usepackage{amsmath}
\usepackage{amsthm}
\usepackage{amssymb}
\usepackage{enumerate}
\usepackage[bookmarks=false,
            pdftitle={Fusion systems on bicyclic 2-groups},
            pdfauthor={Benjamin Sambale},
            pdfkeywords={Fusion systems, bicyclic 2-groups},
            pdfstartview={FitH}]{hyperref}
\newtheorem{Theorem}{Theorem}[section]
\newtheorem{Lemma}[Theorem]{Lemma}
\newtheorem{Proposition}[Theorem]{Proposition}
\newtheorem{Corollary}[Theorem]{Corollary}
\newtheorem{Definition}[Theorem]{Definition}

\renewcommand{\phi}{\varphi}
\newcommand{\C}{\operatorname{C}}
\newcommand{\N}{\operatorname{N}}
\newcommand{\Z}{\operatorname{Z}}
\newcommand{\K}{\operatorname{K}}
\newcommand{\Aut}{\operatorname{Aut}}
\newcommand{\Out}{\operatorname{Out}}
\newcommand{\Inn}{\operatorname{Inn}}
\newcommand{\pcore}{\operatorname{O}}
\newcommand{\GL}{\operatorname{GL}}
\newcommand{\SL}{\operatorname{SL}}
\newcommand{\PSL}{\operatorname{PSL}}
\newcommand{\PGL}{\operatorname{PGL}}

\newcommand{\Ker}{\operatorname{Ker}}
\newcommand{\Hom}{\operatorname{Hom}}
\newcommand{\PSU}{\operatorname{PSU}}

\newcommand{\Syl}{\operatorname{Syl}}

\newcommand{\nunlhd}{\ntrianglelefteq}

\mathchardef\ordinarycolon\mathcode`\:  %defines a nice ":=" 
 \mathcode`\:=\string"8000
 \begingroup \catcode`\:=\active
   \gdef:{\mathrel{\mathop\ordinarycolon}}
 \endgroup

\title{Fusion systems on bicyclic $2$-groups}
\author{Benjamin Sambale}
\date{\today}

\begin{document}
\frenchspacing
\maketitle
\begin{abstract}\noindent
We classify all (saturated) fusion systems on bicyclic $2$-groups. Here, a bicyclic group is a product of two cyclic subgroups. This extends previous work on fusion systems on metacyclic $2$-groups (see [Craven-Glesser, 2012] and [Sambale, 2012]). 
%Along the way we develop some useful results about fusion systems on arbitrary $2$-groups. 
As an application we prove Olsson's Conjecture for all blocks with bicyclic defect groups.
\end{abstract}

\section{Introduction}

Fusion systems occur in many areas of mathematics: group theory, representation theory, topology. This makes it interesting to classify fusion systems on a given family of finite $p$-groups. In particular it is of general interest to find so-called \emph{exotic} fusion systems, i.\,e. fusion systems which do not occur among finite groups (see \cite{OliverVentura} for example). On the other hand, it is often useful to know which $p$-groups admit only nilpotent (sometimes called trivial) fusion systems, i.\,e. fusion systems coming from $p$-groups. 
One family of $p$-groups which comes quickly to mind is the class of metacyclic $p$-groups. Here for odd primes $p$ it is known by work of Stancu \cite{Stancu} that every fusion system is controlled. This means one can classify these fusion systems by looking at $p'$-subgroups of the outer automorphism group and their action. In particular only nonexotic fusion systems occur. The fusion systems on metacyclic $2$-groups were determined in \cite{Sambale}. In this case the $2$-groups of maximal class play an important role. 

In order to generalize these results we consider $p$-groups $P$ which can be written in the form $P=\langle x\rangle\langle y\rangle$ for some $x,y\in P$. We call these groups \emph{bicyclic}. For odd primes $p$, Huppert showed in \cite{Huppertbi} that the class of bicyclic groups coincides with the class of metacyclic groups (see also Satz III.11.5 \cite{Huppert}). 
%This means we can always assume that $\langle x\rangle$ or $\langle y\rangle$ is normal in $P$. 
He also pointed out that this is not true for $p=2$. A prominent counterexample is the wreath product $C_4\wr C_2$. So the aim of the present paper is to classify fusion systems on the wider class of bicyclic $2$-groups. 

Apart from Huppert's work, there are many other contributions to the theory of bicyclic $2$-groups. We mention some of them: \cite{Blackburnbi,Itobi,ohara1,ohara2}. One of these early results is the following: Let $P$ be a nonmetacyclic, bicyclic $2$-group. Then the commutator subgroup $P'$ is abelian of rank at most $2$ and $P/P'$ contains a cyclic maximal subgroup. Moreover, if $P/P'$ has exponent at least $8$, then also $P'$ contains a cyclic maximal subgroup.

Recently, Janko \cite{Janko} presented all bicyclic $2$-groups by generators and relations using an equivalent property (see Theorem~\ref{Janko} below). However, the classification of the bicyclic $2$-groups is not complete, since in Janko's presentation it is not clear if some of the parameters give isomorphic groups. Even more recent results which deal with an application to bipartite graphs can be found in \cite{bipartite}.

If not explicitly stated otherwise, all groups in this paper are finite, and all fusion systems are saturated. In the second section we prove some general results about fusion systems on $p$-groups which are more or less consequences of Alperin's Fusion Theorem. 
%Then we restrict us to the case $p=2$. Here a result of Bender \cite{Bender} allows us to give useful information on so-called essential subgroups. 
After that we consider fusion systems on bicyclic $2$-groups.
Here we obtain the unexpected result that every fusion system on a bicyclic $2$-group $P$ is nilpotent unless $P'$ is cyclic. Conversely, every bicyclic, nonmetacyclic $2$-group with cyclic commutator subgroup provides a nonnilpotent fusion system. All these groups are cyclic extensions of (possibly abelian) dihedral or quaternion groups and their number grows with the square of the logarithm of their order.
Moreover, it turns out that no exotic fusion system shows up here. In fact we construct these fusion systems as fusion systems of cyclic extension of finite groups of Lie type. 
The complete classification is given in Theorem~\ref{main}. As a byproduct, we also investigate the isomorphism problem of some of the groups in Janko's paper \cite{Janko}. 
At the end we prove as an application that Olsson's Conjecture of block theory holds for all blocks with bicyclic defect groups. Other conjectures for blocks with bicyclic defect groups have been investigated in a separate paper \cite{Sambalefurther}.

Most of our notation is standard. A finite $p$-group $P$ has \emph{rank} $r$ if $|P:\Phi(P)|=p^r$, i.\,e. $P$ is generated by $r$ elements, but not by fewer. Similarly the \emph{$p$-rank} of $P$ is the maximal rank of an abelian subgroup of $P$. We denote the members of the lower central series of a $p$-group $P$ by $\K_i(P)$; in particular $\K_2(P)=P'$. Moreover, $\Omega_i(P)=\langle x\in P: x^{p^i}=1\rangle$ and $\mho_i(P):=\langle x^{p^i}:x\in P\rangle$ for $i\ge 1$. For convenience we write $\Omega(P):=\Omega_1(P)$ and $\mho(P):=\mho_1(P)$.
A cyclic group of order $n\in\mathbb{N}$ is denoted by $C_n$. Moreover, we set $C_n^k:=C_n\times\ldots\times C_n$ ($k$ factors). In particular groups of the form $C_n^2$ are called \emph{homocyclic}. A dihedral (resp. semidihedral, quaternion) group of order $2^n$ is denoted by $D_{2^n}$ (resp. $SD_{2^n}$, $Q_{2^n}$). A group $G$ is \emph{minimal nonabelian} if $G$ is nonabelian, but all proper subgroups of $G$ are abelian.
We say that a $p$-group $P$ is minimal nonabelian of \emph{type} $(r,s)$ if
\begin{equation}\label{mnatype}
P\cong\langle x,y\mid x^{p^r}=y^{p^s}=[x,y]^p=[x,x,y]=[y,x,y]=1\rangle
\end{equation}
where $[x,y]:=xyx^{-1}y^{-1}$ and $[x,y,z]:=[x,[y,z]]$ (see \cite{Redei}). Moreover, we set $^xy:=xyx^{-1}$ for elements $x$ and $y$ of a group. A group extension with normal subgroup $N$ is denoted by $N.H$. If the extension splits, we write $N\rtimes H$ for the semidirect product. A central product is denoted by $N\mathop{\ast}H$ where it will be always clear which subgroup of $\Z(N)$ is merged with a subgroup of $\Z(H)$. 
For the language of fusion systems we refer to \cite{Linckelmann}. 

\section{General results}
We begin with two elementary lemmas about minimal nonabelian groups.

\begin{Lemma}\label{charmna}
A finite $p$-group $P$ is minimal nonabelian if and only if $P$ has rank $2$ and $|P'|=p$. 
\end{Lemma}
\begin{proof}
Assume first that $P$ is minimal nonabelian. Choose two noncommuting elements $x,y\in P$. Then $\langle x,y\rangle$ is nonabelian and $P=\langle x,y\rangle$ has rank $2$. Every Element $x\in P$ lies in a maximal subgroup $M\le P$. Since $M$ is abelian, $M\subseteq\C_P(x)$. In particular, all conjugacy classes of $P$ have length at most $p$. By a result of Knoche (see Aufgabe III.24b) in \cite{Huppert}) we obtain $|P'|=p$. 

Next, suppose that $P$ has rank $2$ and $|P'|=p$. Then $P'\le\Z(P)$. For $x,y\in P$ we have $[x^p,y]=[x,y]^p=1$. Hence, $\Phi(P)=P'\langle x^p:x\in P\rangle\le\Z(P)$. For any maximal subgroup $M\le P$ it follows that $|M:\Z(P)|\le|M:\Phi(P)|=p$. Therefore, $M$ is abelian and $P$ is minimal nonabelian.
\end{proof}

\begin{Lemma}
Let $P$ be a minimal nonabelian group of type $(r,s)$. Then the following holds:
\begin{enumerate}[(i)]
\item $|P|=p^{r+s+1}$.
\item $\Phi(P)=\Z(P)=\langle x^2,y^2,[x,y]\rangle\cong C_{p^{r-1}}\times C_{p^{s-1}}\times C_p$.
\item $P'=\langle [x,y]\rangle\cong C_p$.
\end{enumerate}
\end{Lemma}
\begin{proof}
straightforward.
\end{proof}

By Alperin's Fusion Theorem, the morphisms of a fusion system $\mathcal{F}$ on a $p$-group $P$ are controlled by the $\mathcal{F}$-essential subgroups of $P$. 

\begin{Definition}
A subgroup $Q\le P$ is called \emph{$\mathcal{F}$-essential} if the following properties hold:
\begin{enumerate}[(i)]
\item $Q$ is \emph{fully $\mathcal{F}$-normalized}, i.\,e. $\lvert\N_P(R)\rvert\le\lvert\N_P(Q)\rvert$ if $R\le P$ and $Q$ are $\mathcal{F}$-isomorphic.
\item $Q$ is \emph{$\mathcal{F}$-centric}, i.\,e. $\C_P(R)=\Z(R)$ if $R\le P$ and $Q$ are $\mathcal{F}$-isomorphic.
\item $\Out_{\mathcal{F}}(Q):=\Aut_{\mathcal{F}}(Q)/\Inn(Q)$ contains a strongly $p$-embedded subgroup $H$, i.\,e. $p\mid|H|<\lvert\Out_{\mathcal{F}}(Q)\rvert$ and $p\nmid|H\cap {^xH}|$ for all $x\in\Out_{\mathcal{F}}(Q)\setminus H$.
\end{enumerate}
\end{Definition}

Notice that in \cite{Linckelmann} the first property is not required. It should be pointed out that there are usually very few $\mathcal{F}$-essential subgroups. In many cases there are none. For convenience of the reader we state a version of Alperin's Fusion Theorem. For this let $\mathcal{E}$ be a set of representatives for the $\Aut_{\mathcal{F}}(P)$-conjugacy classes of $\mathcal{F}$-essential subgroups of $P$. %GorensteinLyons,

\begin{Theorem}[Alperin's Fusion Theorem]
Let $\mathcal{F}$ be a fusion system on a finite $p$-group $P$. Then every isomorphism in $\mathcal{F}$ is a composition of finitely many isomorphisms of the form $\phi:S\to T$ such that $S,T\le Q\in\mathcal{E}\cup\{P\}$ and there exists $\psi\in\Aut_{\mathcal{F}}(Q)$ with $\psi_{|S}=\phi$. Moreover, if $Q\ne P$, we may assume that $\psi$ is a $p$-element.
\end{Theorem}
\begin{proof}
Apart from the last sentence, this is Theorem~5.2 in \cite{Linckelmann}. Thus for $S\in\mathcal{E}$ and $\phi\in\Aut_{\mathcal{F}}(S)$ we need to show that $\phi$ can be written as a composition of isomorphisms in the stated form. As $S<P$, also $S<\N_P(S)$, so by induction on $|P:S|$ we can assume that the claim is true for any $\mathcal{F}$-automorphism of $\N_P(S)$.
%
%This is a slightly stronger version as in \cite{Linckelmann}. First we show that we can replace an $\mathcal{F}$-essential subgroup $Q$ by $\alpha(Q)$ for $\alpha\in\Aut_{\mathcal{F}}(P)$. For this let $S,T\le Q$, $\psi\in\Aut_{\mathcal{F}}(Q)$ and $\psi_{|S}=\phi:S\to T$. Then $\alpha\psi\alpha^{-1}\in\Aut_{\mathcal{F}}(\alpha(Q))$. Hence $\phi=\alpha^{-1}\circ(\alpha\psi\alpha^{-1})_{|\alpha(S)}\circ\alpha_{|S}$ is a composition of isomorphisms which have the desired form.
%
%In order the prove the last claim, we adapt the proof in \cite{Linckelmann}. For this let $\phi:S\to T$ a isomorphism in $\mathcal{F}$. We show by induction on $|P:S|$ that $\phi$ can be written as a composition of isomorphisms of the stated form. In the case $S=P=T$ there is nothing to do. Thus, let $S<P$. As in \cite{Linckelmann} we may assume that $S=T\in\mathcal{E}$. 
Let $K:=\langle f\in\Aut_{\mathcal{F}}(S)\ \text{$p$-element}\rangle\unlhd\Aut_{\mathcal{F}}(S)$. Since $\Aut_P(S)$ is a Sylow $p$-subgroup of $\Aut_{\mathcal{F}}(S)$, the Frattini argument implies $\Aut_{\mathcal{F}}(S)=K\N_{\Aut_{\mathcal{F}}(S)}(\Aut_P(S))$. Hence, we can write $\phi=\alpha\beta$ such that $\alpha\in K$ and $\beta\in\N_{\Aut_{\mathcal{F}}(S)}(\Aut_P(S))$. With the notation of \cite{Linckelmann} we have $\N_{\beta}=\N_P(S)$. Then $\beta$ can be extended to a morphism $\beta'$ on $\N_P(S)$. Since $S<\N_P(S)$, induction shows that $\beta'$ is a composition of isomorphisms of the stated form and so is $\beta=\beta'_{|S}$ and $\beta^{-1}$. Thus after replacing $\phi$ by $\phi\circ\beta^{-1}$, we may assume $\phi\in K$. Then it is obvious that $\phi$ is a composition of isomorphisms as desired. 
\end{proof}

We deduce some necessary conditions for a subgroup $Q\le P$ in order to be $\mathcal{F}$-essential. Since $Q$ is $\mathcal{F}$-centric, we have $\C_P(Q)\subseteq Q$. 
Since $\Out_{\mathcal{F}}(Q)$ contains a strongly $p$-embedded subgroup, $\Out_{\mathcal{F}}(Q)$ is not a $p$-group and not a $p'$-group. Moreover, $\N_P(Q)/Q$ is isomorphic to a Sylow $p$-subgroup of $\Out_{\mathcal{F}}(Q)$. This shows $Q<P$. We also have $\pcore_p(\Aut_{\mathcal{F}}(Q))=\Inn(Q)$. Consider the canonical homomorphism
\[F:\Aut_{\mathcal{F}}(Q)\to\Aut_{\mathcal{F}}(Q/\Phi(Q)).\]
It is well known that $\Ker F$ is a $p$-group. On the other hand $\Inn(Q)$ acts trivially on the abelian group $Q/\Phi(Q)$. This gives $\Ker F=\Inn(Q)$ and $\Out_{\mathcal{F}}(Q)\cong\Aut_{\mathcal{F}}(Q/\Phi(Q))$. 
%In particular $\N_P(Q)/Q$ acts faithfully on $Q/\Phi(Q)$. Hence, $[\langle x\rangle,Q]\nsubseteq\Phi(Q)$ for all $x\in\N_P(Q)\setminus Q$.
%
%\begin{Proposition}\label{uppertriangular}
%Let $\mathcal{F}$ be a fusion system on a finite $p$-group $P$. 
%If $Q\le P$ is $\mathcal{F}$-essential of rank $r$, then $\Out_{\mathcal{F}}(Q)\cong\Aut_{\mathcal{F}}(Q/\Phi(Q))\le\GL(r,p)$ and $\lvert\N_P(Q)/Q\rvert\le p^{r(r-1)/2}$. Moreover, $\N_P(Q)/Q$ has nilpotency class at most $r-1$ and exponent at most $p^{\lceil\log_p(r)\rceil}$. In particular $\lvert\N_P(Q)/Q\rvert=p$ if $r=2$.
%\end{Proposition}
%\begin{proof}
%A Sylow $p$-subgroup of $\GL(r,p)$ is given by the group $U$ of upper triangular matrices with ones on the main diagonal. We may assume $\N_P(Q)/Q\le U$. Then $U$ has order $p^{r(r-1)/2}$ and nilpotency class $r-1$ (see §III.16 in \cite{Huppert}). Let $x\in U$ and $m:=\lceil\log_p(r)\rceil$. Then in the ring of $r\times r$ matrices over $\mathbb{F}_p$ we have
%\[x^{p^m}-1=(x-1)^{p^m}=0\]
%(using the Frobenius automorphism of $\mathbb{F}_p$).
%This shows that $U$ has exponent at most $p^m$.
%\end{proof}
%
%If $p$ is odd or $Q$ is abelian, a similar argument shows that $\Out_{\mathcal{F}}(Q)$ is isomorphic to a quotient of $\Aut(\Omega(Q))$ (see Theorem~5.2.4 and 5.3.10 in \cite{Gorenstein}). In this case we have $\Omega(Q)\nsubseteq\Z(\N_P(Q))$. In the general case one can replace $\Omega(Q)$ by a so-called “critical” subgroup (see Theorem~5.3.11 in \cite{Gorenstein}).
%
\begin{Lemma}\label{maxclass}
Let $\mathcal{F}$ be a fusion system on a finite $p$-group $P$, and let $Q\le P$ be $\mathcal{F}$-essential.
If $|Q|\le p^2$ or if $Q$ is nonabelian of order $p^3$, then $P$ has maximal class.
\end{Lemma}
\begin{proof}
This follows from Proposition~1.8 and Proposition~10.17 in \cite{Berkovich1}.
\end{proof}

Now we turn to $2$-groups.

\begin{Lemma}\label{autqfess}
Let $\mathcal{F}$ be a fusion system on a finite $2$-group $P$. If $Q\le P$ is an $\mathcal{F}$-essential subgroup of rank at most $3$, then $\Out_{\mathcal{F}}(Q)\cong S_3$ and $\lvert\N_P(Q):Q\rvert=2$.
\end{Lemma}
\begin{proof}
By the remark above we have $\Out_{\mathcal{F}}(Q)\le\GL(r,2)$ where $r$ is the rank of $Q$.
Hence, we may assume that $Q$ has rank $3$. Then $\Out_{\mathcal{F}}(Q)\le\GL(3,2)$. A computer calculation (which of course, can be carried out by hand as well) shows that $S_3$ is the only subgroup of $\GL(3,2)$ (up to isomorphism) with a strongly $2$-embedded subgroup.
%By Theorem~\ref{sylow}, $\Out_{\mathcal{F}}(Q)$ has cyclic Sylow $2$-subgroups. 
%By Theorem~\ref{bender}, $\GL(3,2)$ does not contain a strongly $2$-embedded subgroup. Hence $\Out_{\mathcal{F}}(Q)$ is solvable and has cyclic Sylow $2$-subgroups. 
%In particular $\pcore_{2'}(\Out_{\mathcal{F}}(Q))$ has order $3$, $7$ or $21$. Since the normalizer of a Sylow $7$-subgroup of $\GL(3,2)$ has order $21$, it follows that $\lvert\pcore_{2'}(\Out_{\mathcal{F}}(Q))\rvert\ne 7$. Since this normalizer is selfnormalizing in $\GL(3,2)$, we also have $\lvert\pcore_{2'}(\Out_{\mathcal{F}}(Q))\rvert\ne 21$. This shows $\lvert\pcore_{2'}(\Out_{\mathcal{F}}(Q))\rvert=3$. Now the claim follows easily from $\pcore_2(\Out_{\mathcal{F}}(Q))=1$.
\end{proof}

\begin{Proposition}\label{rank2}
Let $\mathcal{F}$ be a fusion system on a finite $2$-group $P$. If $Q\le P$ is an $\mathcal{F}$-essential subgroup of rank $2$, then one of the following holds: %$P$ has rank $2$ and 
\begin{enumerate}[(i)]
\item $Q\cong C_2^2$ and $P\in\{D_{2^n},SD_{2^n}\}$ for some $n\ge 3$.
\item $Q\cong Q_8$ and $P\in\{Q_{2^n},SD_{2^n}\}$ for some $n\ge 3$.
\item\label{wreath} $Q\cong C_{2^r}^2$ and $P\cong C_{2^r}\wr C_2$ for some $r\ge 2$.
\item\label{general} %$Q\unlhd P$ and
$Q/\Phi(Q')\K_3(Q)$ is minimal nonabelian of type $(r,r)$ for some $r\ge 2$.
%$\N_P(Q)/\Phi(Q')\K_3(Q)\cong\langle x,y,a\mid x^{2^r}=[x,y]^2=[x,x,y]=[y,x,y]=1,\ a^2\in\langle[x,y]\rangle,\ ^ax=y\rangle$
%for some $r\ge 2$.
%In particular for a fixed $r$ there are two possible isomorphism classes for $\N_P(Q)/\Phi(Q')\K_3(Q)$, both of order $2^{2(r+1)}\ge 2^6$. Moreover, $\N_P(Q)/Q'\cong C_{2^r}\wr C_2$ and $Q/\Phi(Q')\K_3(Q)$ is minimal nonabelian of type $(r,r)$. If $\lvert\Phi(Q')\K_3(Q)\rvert=2$, then
%\[Q\cong\begin{cases}
%\langle x,y\mid x^8=[x,x,y]=[y,x,y]=1,\ x^4=y^4=[x,y]^2\rangle&\text{if }r=2,\\
%\langle x,y\mid x^{2^r}=y^{2^r}=[x,y]^4=[x,x,y]=[y,x,y]=1\rangle&\text{otherwise}.
%\end{cases}\]
\end{enumerate}
\end{Proposition}
\begin{proof}
By Lemma~\ref{autqfess} we have $\lvert\N_P(Q):Q\rvert=2$.
If $Q$ is metacyclic, then we have $Q\cong Q_8$ or $Q\cong C_{2^r}^2$ for some $r\in\mathbb{N}$ by Lemma~1 in \cite{Mazurov}. Then for $|Q|\le 8$ the result follows from Lemma~\ref{maxclass}. Thus, assume $Q\cong C_{2^r}^2$ for some $r\ge 2$. Let $g\in\N_P(Q)\setminus Q$. Since $g$ acts nontrivially on $Q/\Phi(Q)$, we may assume $^gx=y$ and $^gy=x$ for $Q=\langle x,y\rangle$. We can write $g^2=(xy)^i$ for some $i\in\mathbb{Z}$, because $g$ centralizes $g^2$. Then an easy calculation shows that $gx^{-i}$ has order $2$. Hence, $\N_P(Q)\cong C_{2^r}\wr C_2$. Since $Q$ is the only abelian maximal subgroup in $\N_P(Q)$, we also have $Q\unlhd\N_P(\N_P(Q))$ and $\N_P(Q)=P$ follows.

Now consider the case where $Q$ is nonmetacyclic. Then $Q$ is also nonabelian. By Hilfssatz~III.1.11c) in \cite{Huppert} we know that $Q'/\K_3(Q)$ is cyclic. In particular $Q'/\Phi(Q')\K_3(Q)$ has order $2$. By Lemma~\ref{charmna}, $\overline{Q}:=Q/\K_3(Q)\Phi(Q')$ is minimal nonabelian.
%Therefore, the group $\overline{Q}:=Q/\K_3(Q)\Phi(Q')$ satisfies $|\overline{Q}'|=p$ and $\overline{Q}'\subseteq\Z(\overline{Q})$. Hence, $[x^2,y]=[x,y]^2=1$ for all $x,y\in\overline{Q}$. It follows that $\Phi(\overline{Q})=\langle x^2:x\in\overline{Q}\rangle\subseteq\Z(\overline{Q})$. Since $\overline{Q}$ has rank $2$, this implies that $\overline{Q}$ is minimal nonabelian. 
The case $\overline{Q}\cong Q_8$ is impossible, because $Q$ does not have maximal class (Taussky's Theorem, see Satz~III.11.9 in \cite{Huppert}).
Let $\alpha$ be an automorphism of $Q$ of order $3$. Since $\alpha$ acts nontrivially on $\overline{Q}$, Lemma~2.2 in \cite{Sambalemna} implies that $\overline{Q}$ is of type $(r,r)$ for some $r\ge 2$. 
\end{proof}

The fusion systems in the first three parts of Proposition~\ref{rank2} are determined in \cite{CravenGlesser} (see also Theorem~\ref{main} below).
Notice that we have not proved that case~\eqref{general} actually occurs. However, calculations with GAP \cite{GAP4} show that there are at least small examples and it is reasonable that many examples exist for arbitrary $r\ge 2$. However, we have no example for case~\eqref{general} where $Q\nunlhd P$. 
%In fact one can show that $Q\unlhd P$ whenever $r=2$. 
%In case $\lvert\Phi(Q')\K_3(Q)\rvert>2$ the group $Q$ is not necessarily unique anymore. 
%If in case~\eqref{general} $Q$ is the only $\mathcal{F}$-essential subgroup, then $\mathcal{F}$ is the fusion system of a finite group with Sylow $2$-subgroup $P$ (see Theorem~4.6 in \cite{Linckelmann}).

%Since every subgroup of a metacyclic group is generated by (at most) two elements, the fusion systems on the metacyclic $2$-groups can easily be determined with Proposition~\ref{rank2}.

%We remark further that all candidates for essential subgroups of $2$-rank $2$ were determined in \cite{CravenGlesser}.
%if r=2 and |R|=2, then Q=[64,19] (GAP)

\begin{Lemma}\label{rank3}
Let $\mathcal{F}$ be a fusion system on a finite $2$-group $P$. If $Q\le P$ is an $\mathcal{F}$-essential subgroup of rank $3$, then $\N_P(Q)/\Phi(Q)\cong D_8\times C_2$ or $\N_P(Q)/\Phi(Q)$ is minimal nonabelian of type $(2,1)$.
%
%one of the following holds:
%\begin{enumerate}[(i)]
%\item $Q\unlhd P$ and $P/\Phi(Q)$ is minimal nonabelian of type $(2,1)$; in particular $P$ has rank $2$.
%\item $\N_P(Q)/\Phi(Q)\cong D_8\times C_2$; in particular $\N_P(Q)$ has rank $3$.
%\end{enumerate}
\end{Lemma}
\begin{proof}
By Lemma~\ref{autqfess} we have $\lvert\N_P(Q):Q\rvert=2$. Since $\N_P(Q)$ acts nontrivially on $Q/\Phi(Q)$, we conclude that $\N_P(Q)/\Phi(Q)$ is nonabelian. One can check that there are only two nonabelian groups of order $16$ with an elementary abelian subgroup of order $8$. The claim follows.
%Then $\N_P(Q)/\Phi(Q)$ is minimal nonabelian of type $(2,1)$ or $\N_P(Q)/\Phi(Q)\cong D_8\times C_2$, because $\N_P(Q)/\Phi(Q)$ contains an elementary abelian subgroup of order $8$. 
%Moreover, $\N_P(Q)$ has rank at most $3$, since $\Phi(Q)\subseteq\Phi(\N_P(Q))$.
%
%If $\N_P(Q)/\Phi(Q)$ is minimal nonabelian, $Q/\Phi(Q)$ is characteristic in $\N_P(Q)/\Phi(Q)$ and $\N_P(Q)=P$ follows.
%In the other case there is nothing to show.
\end{proof}

\section{Bicyclic $2$-groups}

Janko gave the following characterization of bicyclic $2$-groups (see \cite{Janko} or alternatively §87 in \cite{Berkovich2}).
Notice that Janko defines commutators in \cite{Janko} differently than we do. 

\begin{Theorem}[Janko]\label{Janko}
A nonmetacyclic $2$-group $P$ is bicyclic if and only if $P$ has rank $2$ and contains exactly one nonmetacyclic maximal subgroup. 
\end{Theorem}

Using this, he classified all bicyclic $2$-groups in terms of generators and relations. However, it is not clear if different parameters in his paper give nonisomorphic groups. In particular the number of isomorphism types of bicyclic $2$-groups is unknown.

As a corollary of Theorem~\ref{Janko} we obtain the structure of the automorphism group of a bicyclic $2$-group.

\begin{Proposition}\label{biaut}
Let $P$ be a bicyclic $2$-group such that $\Aut(P)$ is not a $2$-group. Then $P$ is homocyclic or a quaternion group of order $8$. In particular $P$ is metacyclic.
\end{Proposition}
\begin{proof}
By Lemma~1 in \cite{Mazurov} we may assume that $P$ is nonmetacyclic. Since $P$ has rank $2$, every nontrivial automorphism of odd order permutes the maximal subgroups of $P$ transitively. By Theorem~\ref{Janko} such an automorphism cannot exists.
\end{proof}

As another corollary of Theorem~\ref{Janko} we see that every subgroup of a bicyclic $2$-group contains a metacyclic maximal subgroup. Since quotients of bicyclic groups are also bicyclic, it follows that every section of a bicyclic $2$-group has rank at most $3$. This will be used in the following without an explicit comment. Since here and in the following the arguments are very specific (i.\,e. not of general interest), we will sometimes apply computer calculations in order to handle small cases.

\begin{Proposition}\label{rank2ess}
Let $\mathcal{F}$ be a fusion system on a bicyclic, nonmetacyclic $2$-group $P$. Suppose that $P$ contains an $\mathcal{F}$-essential subgroup $Q$ of rank $2$. Then $Q\cong C_{2^m}^2$ and $P\cong C_{2^m}\wr C_2$ for some $m\ge 2$. Moreover, $\mathcal{F}=\mathcal{F}_P(C_{2^m}^2\rtimes S_3)$ or $\mathcal{F}=\mathcal{F}_P(\PSL(3,q))$ for some $q\equiv 1\pmod{4}$.
\end{Proposition}
\begin{proof}
By Proposition~\ref{rank2} it suffices for the first claim to show that $Q$ is metacyclic, since minimal nonabelian groups of type $(m,m)$ for $m\ge 2$ are nonmetacyclic (see Proposition~2.8 in \cite{Janko}).
Let $M\le P$ be a metacyclic maximal subgroup of $P$. We may assume $Q\nsubseteq M$. Then $M\cap Q$ is a maximal subgroup of $Q$. Since $Q$ admits an automorphism of order $3$, the maximal subgroups of $Q$ are isomorphic. Now the first claim follows from Proposition~2.2 in \cite{Janko}. The fusion systems on $C_{2^m}\wr C_2$ are given by Theorem~5.3 in \cite{CravenGlesser}. Two of them have $C_{2^m}^2$ as essential subgroup.
\end{proof}

It can be seen that the group $C_{2^m}\wr C_2$ is in fact bicyclic.
Observe that Theorem~5.3 in \cite{CravenGlesser} provides another nonnilpotent fusion system on $C_{2^m}\wr C_2$.
For the rest of this paper we consider the case where the bicyclic, nonmetacyclic $2$-group $P$ has no $\mathcal{F}$-essential subgroup of rank $2$. 

\begin{Definition}\label{equal}
Two fusion systems $\mathcal{F}$ and $\mathcal{F}'$ on a finite $p$-group $P$ are \emph{isomorphic} if there is an automorphism $\gamma\in\Aut(P)$ such that
\[\Hom_{\mathcal{F}'}(\gamma(S),\gamma(T))=\gamma(\Hom_{\mathcal{F}}(S,T)):=\{\gamma\circ\phi\circ\gamma^{-1}:\phi\in\Hom_{\mathcal{F}}(S,T)\}\]
for all subgroups $S,T\le P$. 
\end{Definition}
Observe that if $\gamma$ is an inner automorphism of $P$, then $\Hom_{\mathcal{F}}(\gamma(S),\gamma(T))=\gamma(\Hom_{\mathcal{F}}(S,T))$ for all $S,T\le P$. In the following we consider fusion systems only up to isomorphism.

\begin{Proposition}\label{E8normal}
Let $\mathcal{F}$ be a nonnilpotent fusion system on a bicyclic $2$-group $P$. Suppose that $P$ contains an elementary abelian normal subgroup of order $8$. Then $P$ is minimal nonabelian of type $(n,1)$ for some $n\ge 2$ and $C_{2^{n-1}}\times C_2^2$ is the only $\mathcal{F}$-essential subgroup of $P$. Moreover, $\mathcal{F}=\mathcal{F}_P(A_4\rtimes C_{2^n})$ where $C_{2^n}$ acts as a transposition in $\Aut(A_4)\cong S_4$ (thus $A_4\rtimes C_{2^n}$ is unique up to isomorphism).
\end{Proposition}
\begin{proof}
By hypothesis, $P$ is nonmetacyclic.
Suppose first $|P'|=2$. Then $P$ is minimal nonabelian of type $(n,1)$ for some $n\ge 2$ by Theorem~4.1 in \cite{Janko}. 
We show that $P$ contains exactly one $\mathcal{F}$-essential subgroup $Q$. Since $P$ is minimal nonabelian, every selfcentralizing subgroup is maximal. Moreover, $Q$ has rank $3$ by Proposition~\ref{rank2ess}. Hence, $Q=\langle x^2,y,z\rangle\cong C_{2^{n-1}}\times C_2^2$ is the unique nonmetacyclic maximal subgroup of $P$ (notation from \eqref{mnatype} in the introduction). 
We prove that $\mathcal{F}$ is unique up to isomorphism. By Alperin's Fusion Theorem and Proposition~\ref{biaut} it suffices to describe the action of $\Aut_{\mathcal{F}}(Q)$ on $Q$. First of all $P=\N_P(Q)$ acts on only two four-subgroups $\langle y,z\rangle$ and $\langle x^{2^{n-1}}y,z\rangle$ of $Q$ nontrivially.
%interchanges $y$ with $yz$ and $x^2y$ with $x^2yz$. 
Let $\alpha\in\Aut_{\mathcal{F}}(Q)$ of order $3$. Then $\alpha$ is unique up to conjugation in $\Aut(Q)$, since $\langle\alpha\rangle\in\Syl_3(\Aut(Q))$ and $\Aut(Q)$ is not $3$-nilpotent. Hence, $\alpha$ acts on only one four-subgroup $R$ of $Q$. Let $\beta\in P/Q\le\Aut_{\mathcal{F}}(Q)$. Then $(\alpha\beta)(R)=(\beta\alpha^{-1})(R)=\beta(R)=R$, since $\Aut_{\mathcal{F}}(Q)\cong S_3$ by Lemma~\ref{autqfess}. 
Thus, $\Aut_{\mathcal{F}}(Q)$ acts (nontrivially) on $\langle y,z\rangle$ or on $\langle x^{2^{n-1}}y,z\rangle$. It can be seen easily that the elements $x$ and $x^{2^{n-1}}y$ satisfy the same relations as $x$ and $y$. Hence, after replacing $y$ by $x^{2^{n-1}}y$ if necessary, we may assume that $\Aut_{\mathcal{F}}(Q)$ acts on $\langle y,z\rangle$. Since $\C_Q(\alpha)\cong C_{2^{n-1}}$, we see that $x^2y\notin\C_Q(\alpha)$ or $x^2yz\notin\C_Q(\alpha)$. But then both $x^2y,x^2yz\notin\C_Q(\alpha)$, because $\beta(x^2y)=x^2yz$. Hence $\C_Q(\alpha)=\C_Q(\Aut_{\mathcal{F}}(Q))\in\{\langle x^2\rangle,\langle x^2z\rangle\}$. However, $xy$ and $y$ fulfill the same relations as $x$ and $y$. Hence, after replacing $x$ by $xy$ if necessary, we have 
$\C_Q(\Aut_{\mathcal{F}}(Q))=\langle x^2\rangle$. This determines the action of $\Aut_{\mathcal{F}}(Q)$ on $Q$ completely. In particular $\mathcal{F}$ is uniquely determined up to isomorphism. The group $G=A_4\rtimes C_{2^n}$ as described in the proposition has a minimal nonabelian Sylow $2$-subgroup of type $(n,1)$. Since $A_4$ is not $2$-nilpotent, $\mathcal{F}_P(G)$ is not nilpotent. It follows that $\mathcal{F}=\mathcal{F}_P(G)$. 

Now suppose $|P'|>2$. Then Theorem~4.2 in \cite{Janko} describes the structure of $P$. We use the notation of this theorem. 
Let $Q<P$ be $\mathcal{F}$-essential. By Proposition~\ref{rank2ess}, $Q$ has rank $3$. In particular $Q$ is contained in the unique nonmetacyclic maximal subgroup $M:=E\langle a^2\rangle$ of $P$.
Since $\langle a^4,u\rangle=\Z(M)<Q$, it follows that $Q\in\{\langle a^4,u,v\rangle,\langle a^4,a^2v,u\rangle,M\}$. In the first two cases we have $P'=\langle u,z\rangle\subseteq Q\unlhd P$ which contradicts Lemma~\ref{autqfess}. Hence, $Q=M$.
%One can show that $\Phi(M)=\Z(P)=\langle a^4\rangle$. 
Every automorphism of $M$ of order $3$ acts nontrivially on $M/\Phi(M)$ and thus freely on $M/\Z(M)\cong C_2^2$. However, the subgroups $L\le M$ such that $\Z(M)<L<M$ are nonisomorphic. Contradiction.
\end{proof}

It remains to deal with the case where $P$ does not contain an elementary abelian normal subgroup of order $8$. In particular Theorem~4.3 in \cite{Janko} applies.

\begin{Lemma}\label{rank3bi}
Let $\mathcal{F}$ be a fusion system on a bicyclic $2$-group $P$. If $Q\le P$ is $\mathcal{F}$-essential of rank $3$, then one of the following holds:
\begin{enumerate}[(i)]
\item $Q\unlhd P$ and $P/\Phi(Q)$ is minimal nonabelian of type $(2,1)$.
\item $Q\nunlhd P$ and $P/\Phi(Q)\cong D_8\times C_2$.
\end{enumerate}
\end{Lemma}
\begin{proof}
By Lemma~\ref{rank3} we always have that $\N_P(Q)/\Phi(Q)$ is minimal nonabelian of type $(2,1)$ or isomorphic to $D_8\times C_2$. In case $\N_P(Q)=P$ only the first possibilities can apply, since $P$ has rank $2$. Now assume that $Q\nunlhd P$ and $\N_P(Q)/\Phi(Q)$ is minimal nonabelian of type $(2,1)$. Take $g\in\N_P(\N_P(Q))\setminus\N_P(Q)$ such that $g^2\in\N_P(Q)$. Then $Q_1:={^gQ}\ne Q$ and $Q_1\cap Q$ is $\langle g\rangle$-invariant. Moreover, $\Phi(Q)\subseteq\Phi(\N_P(Q))\subseteq Q_1$ and  
\[\lvert\Phi(Q):\Phi(Q)\cap\Phi(Q_1)\rvert= \lvert\Phi(Q_1):\Phi(Q)\cap\Phi(Q_1)\rvert= \lvert\Phi(Q_1)\Phi(Q):\Phi(Q)\rvert=\lvert\Phi(Q_1/\Phi(Q))\rvert=2,\]
since $Q_1/\Phi(Q)$ ($\ne Q/\Phi(Q)$) is abelian of rank $2$.
Hence, $\N_P(Q)/\Phi(Q)\cap\Phi(Q_1)$ is a group of order $32$ of rank $2$ with two distinct normal subgroups of order $2$ such that their quotients are isomorphic to the minimal nonabelian group of type $(2,1)$. It follows that $\N_P(Q)/\Phi(Q)\cap\Phi(Q_1)$ is the minimal nonabelian group of type $(2,2)$ (this can be checked by computer). However, then all maximal subgroups of $\N_P(Q)/\Phi(Q)\cap\Phi(Q_1)$ have rank $3$ which contradicts Theorem~\ref{Janko}. Thus, we have proved that $\N_P(Q)/\Phi(Q)\cong D_8\times C_2$.
\end{proof}

We are in a position to determine all $\mathcal{F}$-essential subgroups of rank $3$ on a bicyclic $2$-group. This is a key result for the rest of the paper. 

\begin{Proposition}\label{biess}
Let $\mathcal{F}$ be a fusion system on a bicyclic $2$-group $P$. If $Q\le P$ is $\mathcal{F}$-essential of rank $3$, then one of the following holds:
\begin{enumerate}[(i)]
\item $Q\cong C_{2^m}\times C_2^2$ for some $m\ge 1$.
\item $Q\cong C_{2^m}\times Q_8$ for some $m\ge 1$.
\item $Q\cong C_{2^m}\mathop{\ast} Q_8$ for some $m\ge 2$.
\end{enumerate}
\end{Proposition}
\begin{proof}
If $P$ contains an elementary abelian normal subgroup of order $8$, then the conclusion holds by Proposition~\ref{E8normal}. Hence, we will assume that there is no such normal subgroup. 
Let $\alpha\in\Out_{\mathcal{F}}(Q)$ of order $3$ (see Lemma~\ref{autqfess}). 
Since $\lvert\Aut(Q)\rvert$ is not divisible by $9$, we can regard $\alpha$ as an element of $\Aut(Q)$ by choosing a suitable preimage. 
We apply \cite{Thomas3} to the group $Q$ (observe that the rank in \cite{Thomas3} is the $p$-rank in our setting). Let $C:=\C_Q(\alpha)$. Suppose first that $C$ has $2$-rank $3$, i.\,e. $m(C)=3$ with the notation of \cite{Thomas3}. 
%In particular, $C$ has rank at least $2$. 
%If $C:=\C_Q(\alpha)$ has only one involution, case \eqref{rank1} applies by Theorem~D in \cite{Thomas3}. Now assume $m(C)=3$ with the notation of \cite{Thomas3}. 
Since $[Q,\alpha]$ is generated by at most three elements, only the first part of Theorem~B in \cite{Thomas3} can occur. In particular $Q\cong Q_8\mathop{\ast} C$. However, this implies that $Q$ contains a subgroup of rank at least $4$. Contradiction. 

Now assume $m(C)=2$. Then Theorem~A in \cite{Thomas3} gives $Q\cong Q_8\mathop{\ast} C$. Let $Z\le\Z(Q_8\times C)=\Phi(Q_8)\times\Z(C)$ such that $Q\cong (Q_8\times C)/Z$. Then $|Z|=2$ and $C$ has rank at most $2$, since $Q$ has rank $3$. Moreover, it follows that $\Omega(\Z(C))\nsubseteq\Phi(C)$ (otherwise: $Z\le\Phi(Q_8)\times\Phi(C)=\Phi(Q_8\times C)$). By Burnside's Basis Theorem, $C\cong C_2\times C_{2^m}$ is abelian and $Q\cong Q_8\times C_{2^m}$ for some $m\ge 1$.

Finally suppose that $m(C)\le1$, i.\,e. $C$ is cyclic or quaternion. By Theorem~\ref{Janko}, $\Phi(P)$ is metacyclic. Since $\Phi(Q)\subseteq \Phi(P)$ (Satz~III.3.14 in \cite{Huppert}), also $\Phi(Q)$ is metacyclic. According to the action of $\alpha$ on $\Phi(Q)$ one of the following holds (see Proposition~\ref{biaut}):
\begin{enumerate}[(a)]
\item\label{a} $\Phi(Q)\le C\unlhd Q$.
\item\label{b} $\Phi(Q)\cong Q_8$.
\item\label{c} $\Phi(Q)\cap C=1$ and $\Phi(Q)\cong C_{2^n}^2$ for some $n\ge 1$.
\end{enumerate}
We handle these cases separately. First assume case~\eqref{a}. By 8.2.2(a) in \cite{Kurzweil} we have $|Q:C|=4$ and $\alpha$ acts freely on $Q/C$. On the other hand $\alpha$ acts trivially on $Q/\C_Q(C)$ by 8.1.2(b) in \cite{Kurzweil}. This shows $Q=C\C_Q(C)$. If $C$ is quaternion, then $Q=Q_{2^n}\mathop{\ast}\C_Q(C)$. In particular, $\C_Q(C)$ has rank at most $2$. Thus, a similar argument as above yields $Q\cong Q_{2^n}\times C_{2^m}$. However, this is impossible here, because $\alpha$ would act trivially on $Q/\Phi(Q)$ by the definition of $C$. Hence, $C$ is cyclic and central of index $4$ in $Q$. Since, $Q$ has rank $3$, the exponents of $C$ and $Q$ coincide.
If $Q$ is abelian, we must have $Q\cong C_{2^m}\times C_2^2$ for some $m\ge 1$. 
Now assume that $Q$ is nonabelian. Write $C=\langle a\rangle$ and choose $b,c\in Q$ such that $Q/C=\langle bC,cC\rangle$. Since $\langle b\rangle C$ is abelian and noncyclic, we may assume $b^2=1$. Similarly $c^2=1$. Since $Q$ is nonabelian, $^cb\ne b$. Let $|C|=2^m$ where $m\ge 2$. Then $a\in\Z(Q)$ implies $^cb=a^{2^{m-1}}b$. Thus, $Q$ is uniquely determined as
\[Q=\langle a,b,c\mid a^{2^m}=b^2=c^2=[a,b]=[a,c]=1,\ ^cb=a^{2^{m-1}}b\rangle.\]
Since the group $Q_8\mathop{\ast} C_{2^m}\cong D_8\mathop{\ast} C_{2^m}$ has the same properties, we get $Q\cong Q_8\mathop{\ast} C_{2^m}$.

Next we will show that case~\eqref{b} cannot occur for any finite group $Q$. On the one hand we have $Q/\C_Q(\Phi(Q))\le\Aut(Q_8)\cong S_4$. On the other hand $C_2^2\cong\Phi(Q)\C_Q(\Phi(Q))/\C_Q(\Phi(Q))\le\Phi(Q/\C_Q(\Phi(Q)))$. Contradiction.

It remains to deal with case~\eqref{c}. Again we will derive a contradiction. By Theorem~D in \cite{Thomas3}, $C\ne 1$ ($U_{64}$ has rank $4$). The action of $\alpha$ on $Q/\Phi(Q)$ shows $|P:C\Phi(Q)|\ge 4$. Now $\Phi(Q)\cap C=1$ implies $|C|=2$.
%Since $\alpha$ acts freely on $\Phi(Q)$, we get $C\cap\Phi(Q)=1$.
%The facts $\Phi(Q)\cap C=1$ and $|Q:C\Phi(Q)|=4$ imply $|C|=2$. 
There exists an $\alpha$-invariant maximal subgroup $N\unlhd Q$. Thus, $N\cap C\subseteq N\cap C\Phi(Q)\cap C=\Phi(Q)\cap C=1$. In particular we can apply Theorem~D in \cite{Thomas3} which gives $N\cong C_{2^{n+1}}^2$. Hence, $Q\cong N\rtimes C=C_{2^{n+1}}^2\rtimes C_2$ (here $\rtimes$ can also mean $\times$). Choose $x,y\in N$ such that $\alpha(x)=y$ and $\alpha(y)=x^{-1}y^{-1}$. Let $C=\langle c\rangle$. Since $Q$ has rank $3$, $c$ acts trivially on $N/\Phi(N)$. Hence, we find integers $i,j$ such that $^zx=x^iy^j$ and $i\equiv 1\pmod{2}$ and $j\equiv 0\pmod{2}$. Then $^cy=\alpha({^zx})=x^{-j}y^{i-j}$. In particular, the isomorphism type of $Q$ does only depend on $i,j$. Since $c^2=1$, we obtain $i^2-j^2\equiv 1\pmod{2^{n+1}}$ and $j(2i-j)\equiv 0\pmod{2^{n+1}}$. We will show that $j\equiv 0\pmod{2^n}$. This is true for $n=1$. Thus, assume $n\ge 2$. Then $1-j^2\equiv i^2-j^2\equiv 1\pmod{8}$. Therefore, $j\equiv 0\pmod{4}$. Now $j(2i-j)\equiv 0\pmod{2^{n+1}}$ implies $j\equiv 0\pmod{2^n}$. In particular $i^2\equiv i^2-j^2\equiv 1\pmod{2^{n+1}}$. Hence, we have two possibilities for $j$ and at most four possibilities for $i$. This gives at most eight isomorphism types for $Q$. 
Now we split the proof into the cases $Q\unlhd P$ and $Q\nunlhd P$.

Suppose $Q\unlhd P$. Then $|P:Q|=2$ by Lemma~\ref{autqfess}. Moreover, $\Omega(Q)\unlhd P$. Since $P$ does not contain an elementary abelian normal subgroup of order $8$, it follows that $Q$ contains more than seven involutions. With the notation above, let $x^ry^sc$ be an involution such that $x^ry^s\notin\Omega(N)$. Then $1=x^ry^scx^ry^sc=x^{r+ir-js}y^{s+jr+(i-j)s}$ and $r(1+i)-js\equiv s(1+i)+jr-js\equiv 0\pmod{2^{n+1}}$. In case $n=1$ we have $|P|=64$. Here it can be shown by computer that $P$ does not exist. Hence, suppose $n\ge 2$ in the following. Suppose further that $i\equiv 1\pmod{2^n}$. Then we obtain $2r\equiv 2s\equiv 0\pmod{2^n}$. Since $x^ry^s\notin\Omega(N)$ we may assume that $r\equiv\pm 2^{n-1}\pmod{2^{n+1}}$ (the case $s\equiv\pm 2^{n-1}\pmod{2^{n+1}}$ is similar). However, this leads to the contradiction $0\equiv r(1+i)-js\equiv 2^n\pmod{2^{n+1}}$. This shows that $i\equiv -1\pmod{2^n}$. In particular, $x^{i-1}y^i={^cx}x^{-1}=[c,x]\in Q'$ and $x^{-j}y^{i-j-1}=[c,y]\in Q'$. This shows $C_{2^n}^2\cong Q'=\Phi(Q)$. 
%Choose $g\in P\setminus Q$. Then $g$ acts nontrivially on $N/\Phi(Q)$, because $\alpha$ does as well. In particular $N\unlhd P$. 
By Lemma~\ref{rank3bi}, $P/\Phi(Q)$ is minimal nonabelian of type $(2,1)$. Since $Q'\subseteq P'$, we conclude that $P/P'\cong C_4\times C_2$.
%Since $N/\Phi(Q)\ne\Z(P/\Phi(Q))=\Phi(P/\Phi(Q))$, it follows that $P/N=\langle gN\rangle$ is cyclic of order $4$. Hence, $P'\subseteq N$. Write $N=\langle x\rangle\times\langle y\rangle$ such that $^gx=y$. Then $xy^{-1}=[x,g]\in P'$. Thus also $N/P'$ is cyclic. By Theorem~4.3(d) in \cite{Janko} we know that $P/P'\cong C_{2^m}\times C_2$ for some $m\ge 2$. Considering $N$ gives $m=2$. 
% and $P/P'=\langle xP',gP'\rangle$. By Theorem~4.3(d) in \cite{Janko} we know that $P/P'\cong C_{2^m}\times C_2$ for some $m\ge 2$. In particular $\Phi(P/P')=\langle x^2P',g^2P'\rangle$ is cyclic. Since $g^2\not\equiv x^2\pmod{N}$, it follows that $x^2\in P'$. This implies $P'=\langle x^2,xy^{-1}\rangle=\langle x^2,xy\rangle$ and $m=2$. 
%In case $n=1$ we have $|P|=64$. Here it can be shown by computer that $P$ does not exist. Thus, assume $n\ge 2$. 
Then $P$ is described in Theorem~4.11 in \cite{Janko}. In particular $\Phi(P)$ is abelian. Choose $g\in P\setminus Q$. Then $g$ acts nontrivially on $N/\Phi(Q)$, because $\alpha$ does as well. This shows $N\unlhd P$ and $C_2^2\cong N/\Phi(Q)\ne\Z(P/\Phi(Q))=\Phi(P/\Phi(Q))$. Hence, $P/N$ is cyclic and $\Phi(P)\ne N$. Therefore, $Q$ contains two abelian maximal subgroups and $N\cap\Phi(P)\subseteq \Z(Q)$. Now a result of Knoche (see Aufgabe~III.7.24) gives the contradiction $|Q'|=2$. 

%In particular $\Phi(P)=\langle x^2,xy,g^2\rangle$ is abelian and metacyclic. Since $g^4\in\Phi(Q)=\langle x^2,y^2\rangle$ and $g(g^4\mho_2(N))g^{-1}=g^4\mho_2(N)$, we get $g^4\in\langle x^2y^2\rangle\mho_2(N)$. This gives the contradiction $\Phi(P)/\langle x^4,x^2y^2\rangle\cong C_2^3$.

Now assume $Q\nunlhd P$. We will derive the contradiction that $\N_P(Q)$ does not contain a metacyclic maximal subgroup. By Lemma~\ref{rank3bi}, $\N_P(Q)/\Phi(Q)\cong D_8\times C_2$. 
%By way of contradiction we may assume that $\N_P(Q)/\Phi(Q)$ is minimal nonabelian of type $(2,1)$ (Lemma~\ref{rank3}). Take $g\in\N_P(\N_P(Q))\setminus\N_P(Q)$ such that $g^2\in\N_P(Q)$. Then $Q_1:={^gQ}\ne Q$ and $Q_1\cap Q$ is $\langle g\rangle$-invariant. Moreover, $\Phi(Q)\subseteq\Phi(\N_P(Q))\subseteq Q_1$ and  
%\[\lvert\Phi(Q):\Phi(Q)\cap\Phi(Q_1)\rvert= \lvert\Phi(Q_1):\Phi(Q)\cap\Phi(Q_1)\rvert= \lvert\Phi(Q_1)\Phi(Q):\Phi(Q)\rvert=\lvert\Phi(Q_1/\Phi(Q))\rvert=2,\]
%since $Q_1/\Phi(Q)$ ($\ne Q/\Phi(Q)$) is abelian of rank $2$.
%Hence, $\N_P(Q)/\Phi(Q)\cap\Phi(Q_1)$ is a group of order $32$ of rank $2$ with two distinct normal subgroups of order $2$ such that their quotients are isomorphic to the minimal nonabelian group of type $(2,1)$. It follows that $\N_P(Q)/\Phi(Q)\cap\Phi(Q_1)$ is the minimal nonabelian group of type $(2,2)$ (this can be checked by computer). However, then all maximal subgroups of $\N_P(Q)/\Phi(Q)\cap\Phi(Q_1)$ have rank $3$ which contradicts Theorem~\ref{Janko}. Thus, we have proved that $\N_P(Q)/\Phi(Q)\cong D_8\times C_2$.
%Now it is not hard to see that $\N_P(Q)/Q\cap Q_1$ is minimal nonabelian of type $(2,2)$. 
%Then $\N_P(Q)/\Phi(Q)\cong D_8\times C_2$ by Lemma~\ref{rank3}. 
Choose $g\in\N_P(Q)\setminus Q$. Then $g$ acts nontrivially on $N/\Phi(N)$, because $\alpha$ does as well. In particular $N\unlhd\N_P(Q)$. This implies \[g^2\Phi(Q)\in\mho(\N_P(Q)/\Phi(Q))=(\N_P(Q)/\Phi(Q))'\subseteq N/\Phi(Q)\]
and $g^2\in N$. As above, we may choose $x,y\in N$ such that $^gx=y$ and $^gy=x$. Since $g$ centralizes $g^2$, we can write $g^2=(xy)^i$ for some $i\in\mathbb{Z}$. Then $gx^{-i}$ has order $2$. Hence, we may assume that $g^2=1$ and $\langle N,g\rangle\cong C_{2^{n+1}}\wr C_2$. In case $n=1$ we have $\lvert\N_P(Q)\rvert=64$. Here one can show by computer that $\N_P(Q)$ does not exist. Hence, $n\ge 2$. Let $M$ be a metacyclic maximal subgroup of $\N_P(Q)$. Since $\langle \Phi(Q),g\rangle\cong C_{2^n}\wr C_2$ is not metacyclic, we conclude that $g\notin M$. Let $C=\langle c\rangle$. Then $\langle \Phi(Q),c\rangle$ has rank $3$. In particular, $c\notin M$. This leaves two possibilities for $M$. It is easy to see that $\langle N,gc\rangle\cong C_{2^{n+1}}\wr C_2$. Thus, $M=\langle \Phi(Q),xc,gc\rangle$. 
Assume $(gc)^2\in\Phi(Q)$. Then it is easy to see that $\langle\Phi(Q),gc\rangle\cong C_{2^n}\wr C_2$ is not metacyclic. This contradiction shows $(gc)^2\equiv xy\pmod{\Phi(Q)}$. 
%Since $xc$ and $gc$ do not commute modulo $\Phi(Q)$, we get $M/\Phi(Q)$. In particular $xg\Phi(Q)$ and $yg\Phi(Q)$ are the only elements of order $4$ in $M/\Phi(Q)$.
%Let $(gc)^2=x^ry^s$. Since $x^sy^r=g(gc)^2g=(cg)^2=(gc)^{-2}=x^{-r}y^{-s}$, we have $(gc)^2=(xy^{-1})^r$ where $r$ is odd. 
Moreover, $c(gc)^2c=(cg)^2=(gc)^{-2}$. 
%By the calculation above, this implies that $zczc\in\Omega(Q)=\langle x^{2^n},y^{2^n}\rangle$ where $z:=(xy^{-1})^r$. 
Since $N=\langle gc,\alpha(gc)\rangle$, $c$ acts as inversion on $N$. In particular, $(xc)^2=1$. Hence $\langle\Omega(Q),xc\rangle\subseteq M$ is elementary abelian of order $8$. Contradiction.
%We write $N=\langle x\rangle\times\langle y\rangle$ such that $^gx=y$. 
%Let $C=\langle c\rangle$. Then $\langle \Phi(Q),c\rangle$ has rank $3$. On the other hand, $\N_P(Q)$ contains a metacyclic maximal subgroup $M$. Thus, $c\notin M$. Also $M\ne \langle N,g\rangle$. It is easy to see that $\langle N,gc\rangle\cong C_{2^{n+1}}\wr C_2$. Hence, $M\ne \langle N,gc\rangle$. Consider next $M=\langle xy,xc,g\rangle\Phi(Q)$
%Since $xy\Phi(Q)=x^{-1}y\Phi(Q)=[g,x]\Phi(Q)=(P/\Phi(Q))'$, $M:=\langle x^2,xy,c,g\rangle$ is a maximal subgroup of $\N_P(Q)$.
%%Write $cxc=xx^{2i}y^{2j}$. Then $cyc=c\alpha(x)c=\alpha(cxc)=yx^{-2j}y^{2(i-j)}$ and \[x=c^2xc^2=x^{(2i+1)^2}y^{2j(2i+1)}x^{-4j^2}y^{(1+2(i-j))2j}=xx^{4(i+i^2-j^2)}y^{8ij+4j-4j^2}.\]
%It can be seen easily that $K:=\langle x^4,x^2y^2\rangle\unlhd M$. We prove that $c$ and $g$ commute modulo $K$. By the structure of $\N_P(Q)/\Phi(Q)$ we get $cgc^{-1}g^{-1}=(cg)^2=x^{2i}y^{2j}\in\Phi(Q)$ for some $i,j\in\mathbb{Z}$. Moreover, \[x^{2j}y^{2i}=g(cg)^2g=gcgc=(cg)^{-2}=x^{-2i}y^{-2j}\]
%and $(cg)^2=(xy^{-1})^{2i}\in K$. It follows that $M/K\cong C_2^4$. Contradiction.
\end{proof}

Let $Q$ be one of the groups in Proposition~\ref{biess}. Then it can be seen that there is an automorphism $\alpha\in\Aut(Q)$ of order $3$. Since the kernel of the canonical map $\Aut(Q)\to\Aut(Q/\Phi(Q))\cong\GL(3,2)$ is a $2$-group, we have $\langle\alpha\rangle\in\Syl_3(\Aut(Q))$. If $\alpha$ is not conjugate to $\alpha^{-1}$ in $\Aut(Q)$, then Burnside's Transfer Theorem implies that $\Aut(Q)$ is $3$-nilpotent. But then also $\Out_{\mathcal{F}}(Q)\cong S_3$ would be $3$-nilpotent which is not the case. Hence, $\alpha$ is unique up to conjugation in $\Aut(Q)$. In particular the isomorphism type of $\C_Q(\alpha)$ is uniquely determined.

\begin{Proposition}\label{essnormal}
Let $\mathcal{F}$ be a fusion system on a bicyclic $2$-group $P$. If $Q\unlhd P$ is $\mathcal{F}$-essential of rank $3$, then one of the following holds:
\begin{enumerate}[(i)]
\item $P$ is minimal nonabelian of type $(n,1)$ for some $n\ge 2$. 
\item $P\cong Q_8\rtimes C_{2^n}$ for some $n\ge 2$. Here $C_{2^n}$ acts as a transposition in $\Aut(Q_8)\cong S_4$.
%\item\label{wr} $P\cong C_4\wr C_2$. 
\item $P\cong Q_8.C_{2^n}$ for some $n\ge 2$.
\end{enumerate}
In particular $P'$ is cyclic.
\end{Proposition}
\begin{proof}
We use Proposition~\ref{biess}. If $Q$ is abelian, then $C_2^3\cong\Omega(Q)\unlhd P$. By Proposition~\ref{E8normal}, $P$ is minimal nonabelian of type $(n,1)$ for some $n\ge 2$.
Now assume $Q\cong Q_8\times C_{2^{n-1}}$ for some $n\ge 2$. We write $Q=\langle x,y,z\rangle$ such that $\langle x,y\rangle\cong Q_8$ and $\langle z\rangle\cong C_{2^{n-1}}$. Moreover, choose $g\in P\setminus Q$. Let $\alpha\in\Out_{\mathcal{F}}(Q)$ as usual. Then $\alpha$ acts nontrivially on $Q/\Z(Q)\cong C_2^2$ and so does $g$. Hence, we may assume $^gx=y$. Since $g^2\in Q$, it follows that $^gy={^{g^2}x}\in\{x,x^{-1}\}$. By replacing $g$ with $gx$ if necessary, we may assume that $^gy=x$. Hence, $g^2\in\Z(Q)$. By Lemma~\ref{rank3bi}, $P/\Phi(Q)$ is minimal nonabelian of type $(2,1)$. In particular, $Q/\Phi(Q)=\Omega(P/\Phi(Q))$. This gives $g^2\notin\Phi(Q)$ and
%and $^gz=z$. After replacing $g$ by $gx$ if necessary, we obtain $^gy=x$. In particular $g^2$ centralizes $\langle x,y\rangle$. 
%Since $P/\Phi(Q)$ is minimal nonabelian of type $(2,1)$, we get 
$g^2\in z\langle x^2,z^2\rangle$. Since $^g(x^2)=x^2$, we get $^gz=z$.
After replacing $g$ by $gz^i$ for a suitable integer $i$, it turns out that $g^2\in\{z,zx^2\}$. In the latter case we replace $z$ by $x^2z$ and obtain $g^2=z$. Hence, $P=Q_8\rtimes C_{2^n}$ as stated. Moreover, $g$ acts on $\langle x,y\rangle$ as an involution in $\Aut(Q_8)\cong S_4$. Since an involution which is a square in $\Aut(Q_8)$ cannot act nontrivially on $Q_8/\Phi(Q_8)$, $g$ must correspond to a transposition in $S_4$. This describes $P$ up to isomorphism. Since $P=\langle gx\rangle\langle g\rangle$, $P$ is bicyclic. In particular $P'\subseteq\langle x,y\rangle$ is abelian and thus cyclic.

Finally suppose that $Q=Q_8\mathop{\ast}C_{2^n}$ for some $n\ge 2$. We use the same notation as before. In particular $x^2=z^{2^{n-1}}$. The same arguments as above give $g^2=z$ and
\[P=\langle x,y,g\mid x^4=1,\ x^2=y^2=g^{2^n},\ ^yx=x^{-1},\ ^gx=y,\ ^gy=x\rangle\cong Q_8.C_{2^n}.\]
%For $n=2$ it can be seen that $P\cong C_4\wr C_2$. 
Then $P=\langle gx\rangle\langle g\rangle$ is bicyclic and $P'$ cyclic.
\end{proof}

We will construct the groups and fusion systems in the last proposition systematically in our main Theorem~\ref{main}.

Let $\mathcal{F}$ be a fusion system on a $2$-group $P$. Following Definition~3.1 in \cite{Linckelmann} every subgroup $Q\le\Z(P)$ gives rise to another fusion system $\C_{\mathcal{F}}(Q)$ on $P$.

\begin{Definition}
The largest subgroup $Q\le\Z(P)$ such that $\C_{\mathcal{F}}(Q)=\mathcal{F}$ is called the \emph{center} $\Z(\mathcal{F})$ of $\mathcal{F}$. Accordingly, we say $\mathcal{F}$ is \emph{centerfree} if $\Z(\mathcal{F})=1$.
\end{Definition}

%We say that $\mathcal{F}$ is \emph{centerfree} if there is no $1\ne Q\le\Z(P)$ such that $\mathcal{F}=\C_{\mathcal{F}}(Q)$.
The following result is useful to reduce the search for essential subgroups. Notice that the centerfree fusion systems on metacyclic $2$-group are determined in \cite{CravenGlesser}.

\begin{Proposition}\label{centerfree}
Let $\mathcal{F}$ be a centerfree fusion system on a bicyclic, nonmetacyclic $2$-group $P$. Then there exists an abelian $\mathcal{F}$-essential subgroup $Q\le P$ isomorphic to $C_{2^m}^2$ or to $C_{2^m}\times C_2^2$ for some $m\ge 1$. %Moreover, if $P\not\cong C_{2^m}\wr C_2$, then $\Z(P)\cong C_2$.
\end{Proposition}
\begin{proof}
By way of contradiction assume that all $\mathcal{F}$-essential subgroups are isomorphic to $C_{2^m}\times Q_8$ or to $C_{2^m}\mathop{\ast}Q_8$ (use Propositions~\ref{rank2ess} and \ref{biess}). Let $z\in\Z(P)$ be an involution. Since $\Z(\mathcal{F})=1$, Alperin's Fusion Theorem in connection with Theorem~\ref{Janko} implies that there exists an $\mathcal{F}$-essential subgroup $Q\le P$ such that $z\in\Z(Q)$. Moreover, there is an automorphism $\alpha\in\Aut(Q)$ such that $\alpha(z)\ne z$. Of course $\alpha$ restricts to an automorphism of $\Z(Q)$. 
In case $Q\cong C_{2^m}\mathop{\ast}Q_8$ this is not possible, since $\Z(Q)$ is cyclic. Now assume $Q\cong C_{2^m}\times Q_8$. Observe that we can assume that $\alpha$ has order $3$, because the automorphisms in $\Aut_P(Q)$ fix $z$ anyway.
But then $\alpha$ acts trivially on $Q'$ and on $\Omega(Q)/Q'$ and thus also on $\Omega(Q)\ni z$. Contradiction.
\end{proof}

\subsection{The case $P'$ noncyclic}
The aim of this section is to prove that there are only nilpotent fusion system provided $P'$ is noncyclic.
We do this by a case by case analysis corresponding to the theorems in \cite{Janko}. By Proposition~\ref{essnormal} we may assume that there are no normal $\mathcal{F}$-essential subgroups. %In particular we will use Proposition~\ref{essnotnormal}.

Let $\mathcal{F}$ be a nonnilpotent fusion system on the bicyclic $2$-group $P$.
Assume for the moment that $P'\cong C_2^2$. Then $P$ does not contain an elementary abelian subgroup of order $8$ by Proposition~\ref{E8normal}. Hence, Theorem~4.6 in \cite{Janko} shows that $P$ is unique of order $32$. In this case we can prove with computer that there are no candidates for $\mathcal{F}$-essential subgroups. Hence, we may assume $\Phi(P')\ne 1$ in the following. 

We introduce a few notation from Theorem~4.3 in \cite{Janko} that will be used for the rest of the paper:
\begin{align*}
\Phi(P)=P'\langle a^2\rangle=\langle a^2\rangle\langle v\rangle,&& M=E\langle a^2\rangle=\langle x\rangle\langle a^2\rangle\langle v\rangle.
\end{align*}
Here, $M$ is the unique nonmetacyclic maximal subgroup of $P$.

\begin{Proposition}
Let $P$ be a bicyclic $2$-group such that $P'$ is noncyclic and $P/\Phi(P')$ contains no elementary abelian normal subgroup of order $8$. Then every fusion system on $P$ is nilpotent.
\end{Proposition}
\begin{proof}
The case $\Phi(P')=1$ was already handled. So we may assume $\Phi(P')\ne 1$. In particular Theorem~4.7 in \cite{Janko} applies.
%As in the proof of this theorem we can also use Theorem~4.3 in \cite{Janko}. With the notation in this paper we have $\Phi(P)=P'\langle a^2\rangle=\langle a^2\rangle\langle v\rangle$, and $M=E\langle a^2\rangle=\langle x\rangle\langle a^2\rangle\langle v\rangle$ is the unique nonmetacyclic maximal subgroup of $P$. 
Let $\mathcal{F}$ be a nonnilpotent fusion system on $P$. Assume first that there exists an $\mathcal{F}$-essential subgroup $Q\in\{C_{2^m}\times C_2^2,C_{2^m}\mathop{\ast}Q_8\cong C_{2^m}\mathop{\ast}D_8\}$ (the letter $m$ is not used in Theorem~4.7 of \cite{Janko}).
Theorem~4.7 of \cite{Janko} also shows that $\Phi(P)$ is metacyclic and abelian. Since $Q$ contains more than three involutions, there is an involution $\beta\in M\setminus\Phi(P)$. Hence, we can write $\beta=xa^{2i}v^j$ for some $i,j\in\mathbb{Z}$.
Now in case (a) of Theorem~4.7 of \cite{Janko} we derive the following contradiction:
\[\beta^2=xa^{2i}v^jxa^{2i}v^j=xa^{2i}xa^{2i}=x^2(av)^{2i}a^{2i}=x^2a^{2i}u^iz^{\xi i}a^{2i}=uz^{(\eta+\xi)i}\ne 1.\]
Similarly in case (b) we get:
\begin{align*}
\beta^2&=xa^{2i}v^jxa^{2i}v^j=xa^{2i}xz^ja^{2i}=x^2(av)^{2i}z^ja^{2i}=x^2a^{2i}u^iv^{2^{n-2}i}z^{\xi i}z^ja^{2i}\\
&=x^2v^{2^{n-2}i}z^{\eta i}v^{2^{n-2}i}z^{\xi i}z^j=uz^{i(1+\eta+\xi)+j}\ne 1.
\end{align*}

Next assume that there is an $\mathcal{F}$-essential subgroup $C_{2^m}\times Q_8\cong Q\le P$ for some $m\ge 1$. 
Suppose $m\ge 3$ for the moment. Since $Q\subseteq M$, it is easy to see that $M\setminus\Phi(P)$ contains an element of order at least $8$. However, we have seen above that this is impossible. Hence, $m\le 2$. 
By Proposition~\ref{essnormal}, $Q$ is not normal in $P$.
Since $Q<\N_M(Q)\le\N_P(Q)$, we have $\N_P(Q)\le M=\N_P(Q)\Phi(P)$. A computer calculation shows that $\N_P(Q)\cong Q_{16}\times C_{2^m}$. Thus, $\N_P(Q)\cap\Phi(P)\cong C_8\times C_{2^m}$, because $\Phi(P)$ is abelian.
Hence, there exist $\beta=xa^{2^i}y^j\in\N_P(Q)\setminus\Phi(P)\subseteq M\setminus\Phi(P)$ and $\delta\in\N_P(Q)\cap\Phi(P)$ such that $\beta^2=\delta^4$. As above we always have $\beta^2\in u\langle z\rangle$. However, in both cases (a) and (b) we have $\delta^4\in\mho_2(\Phi(P))\cap\Omega(\Phi(P))=\langle a^8\rangle\langle v^{2^{n-1}}\rangle=\langle z\rangle$. Contradiction.
%Finally suppose that $Q\cong Q_8\mathop{\ast}C_{2^m}$ is $\mathcal{F}$-essential for some $m\ge 2$.
%Then $\Z(P)\subseteq\C_P(Q)=\Z(Q)$ is cyclic and only case (b) can occur. By Proposition~\ref{essnotnormal}, $\N_P(Q)\cong Q_{16}\mathop{\ast}C_{2^m}\cong D_{16}\mathop{\ast}C_{2^m}$. In particular there exists an involution $\beta\in\N_P(Q)\setminus\Phi(P)\subseteq M\setminus\Phi(P)$. But this was already excluded. Hence, the claim follows.
\end{proof}

If $P'$ is cyclic, $P/\Phi(P')$ is minimal nonabelian and thus contains an elementary abelian normal subgroup of order $8$. 
Hence, it remains to deal with the case where $P/\Phi(P')$ has a normal subgroup isomorphic to $C_2^3$.

Our next goal is to show that $P'$ requires a cyclic maximal subgroup group in order to admit a nonnilpotent fusion system.

\begin{Proposition}\label{nonhomo}
Let $P$ be a bicyclic $2$-group such that $P'\cong C_{2^r}\times C_{2^{r+s}}$ for some $r\ge 2$ and $s\in\{1,2\}$. Then every fusion system on $P$ is nilpotent.
\end{Proposition}
\begin{proof}
We apply Theorem~4.11 and Theorem~4.12 in \cite{Janko} simultaneously. As usual assume first that $P$ contains an $\mathcal{F}$-essential subgroup $Q\cong C_{2^m}\times C_2^2$ for some $m\ge 1$ ($m$ is not used in the statement of Theorem~4.11 in \cite{Janko}). Then $Q\cap\Phi(Q)\cong C_{2^m}\times C_2$, since $\Phi(P)$ is abelian and metacyclic. We choose $\beta:=xa^{2i}v^j\in Q\setminus\Phi(P)$. In case $m\ge 2$, $\beta$ fixes an element of order $4$ in $Q\cap\Phi(P)$. Since $\Phi(P)$ is abelian, all elements of $\Phi(P)$ of order $4$ are contained in 
\[\Omega_2(\Phi(P))=\begin{cases}
\langle b^{2^{r-2}},v^{2^{r-1}}\rangle&\text{if Theorem~4.11 applies},\\
\langle b^{2^{r-1}},v^{2^{r-1}}\rangle&\text{if Theorem~4.12 applies}.
\end{cases}\]
However, the relations in Theorem~4.11/12 in \cite{Janko} show that $x$ and thus $\beta$ acts as inversion on $\Omega_2(\Phi(P))$. Hence, $m=1$. 
Then $\N_P(Q)\cap\Phi(Q)\cong C_4\times C_2$ by Lemma~\ref{rank3bi}. In particular there exists an element $\rho\in\Omega_2(\Phi(P))\setminus(\N_P(Q)\cap\Phi(Q))$. Then $^\rho\beta=\beta\rho^{-2}\in Q$. Since $Q=\langle\beta\rangle (Q\cap\Phi(P))$, we derive the contradiction $\rho\in\N_P(Q)$.

Next suppose that $Q\cong C_{2^m}\times Q_8$ for some $m\ge 1$. Here we can repeat the argument word by word. Finally the case $Q\cong C_{2^m}\mathop{\ast}Q_8$ cannot occur, since $\Z(P)$ is noncyclic.
\end{proof}

The next lemma is useful in a more general context.

\begin{Lemma}\label{metanormal}
Let $P$ be a metacyclic $2$-group which does not have maximal class. Then every homocyclic subgroup of $P$ is given by $\Omega_i(P)$ for some $i\ge 0$.
\end{Lemma}
\begin{proof}
Let $C_{2^k}^2\cong Q\le P$ with $k\in\mathbb{N}$. We argue by induction on $k$.
By Exercise~1.85 in \cite{Berkovich1}, $C_2^2\cong\Omega(P)$. Hence, we may assume $k\ge 2$. By induction it suffices to show that $P/\Omega(P)$ does not have maximal class. Let us assume the contrary.
Since $P/\Omega(P)$ contains more than one involution, $P/\Omega(P)$ is a dihedral group or a semidihedral group. Let $\langle x\rangle\unlhd P$ such that $P/\langle x\rangle$ is cyclic. Then $\langle x\rangle\Omega(P)/\Omega(P)$ and $(P/\Omega(P))/(\langle x\rangle\Omega(P)/\Omega(P))\cong P/\langle x\rangle\Omega(P)$ are also cyclic. This yields $|P/\langle x\rangle\Omega(P)|=2$ and $|P/\langle x\rangle|=4$. Since $P/\Omega(P)$ is a dihedral group or a semidihedral group, there exists an element $y\in P$ such that the following holds:
\begin{enumerate}[(i)]
\item $P/\Omega(P)=\langle x\Omega(P),y\Omega(P)\rangle$,
\item $y^2\in\Omega(P)$,
\item $^yx\equiv x^{-1}\pmod{\Omega(P)}$ or $^yx\equiv x^{-1+2^{n-2}}\pmod{\Omega(P)}$ with $|P/\Omega(P)|=2^n$ and without loss of generality, $n\ge 4$.
\end{enumerate}
Since $P=\langle x,y\rangle\Omega(P)\subseteq\langle x,y\rangle\Phi(P)=\langle x,y\rangle$, we have shown that $P$ is the semidirect product of $\langle x\rangle$ with $\langle y\rangle$. Moreover \[^yx\in\{x^{-1},\ x^{-1+2^{n-1}},\ x^{-1+2^{n-2}},\ x^{-1-2^{n-2}}\}.\]
Since $Q\cap\langle x\rangle$ and $Q/Q\cap\langle x\rangle\cong Q\langle x\rangle/\langle x\rangle$ are cyclic, we get $k=2$ and $x^{2^{n-2}}\in Q$. But then, $Q$ cannot be abelian, since $n\ge 4$. Contradiction.
\end{proof}

Note that in general for a metacyclic $2$-group $P$ which does not have maximal class it can happen that $P/\Omega(P)$ has maximal class.

\begin{Proposition}\label{nonhomo2}
Let $P$ be a bicyclic $2$-group such that $P'\cong C_{2^r}^2$ for some $r\ge 2$. Then every fusion system on $P$ is nilpotent.
\end{Proposition}
\begin{proof}
We apply Theorem~4.9 in \cite{Janko}. The general argument is quite similar as in Proposition~\ref{nonhomo}, but we need more details. Assume first that $Q\cong C_{2^m}\times C_2^2$ for some $m\ge 1$ is $\mathcal{F}$-essential in $P$ ($m$ is not used in the statement of Theorem~4.9 in \cite{Janko}). Since $\Phi(P)$ has rank $2$, we get $Q\cap\Phi(Q)\cong C_{2^m}\times C_2$. We choose $\beta:=xa^{2i}v^j\in Q\setminus\Phi(P)$. 
Suppose first that $m\ge 2$. Then $\beta$ fixes an element $\delta\in Q\cap\Phi(P)$ of order $4$. Now $\Phi(P)$ is a metacyclic group with $\Omega(\Phi(P))\cong C_2^2$ and $C_4^2\cong\Omega_2(P')\le\Phi(P)$. So Lemma~\ref{metanormal} implies $\Omega_2(\Phi(P))=\langle v^{2^{r-2}},b^{2^{r-2}}\rangle\cong C_4^2$. In case $r=2$ we have $|P|=2^7$, and the claim follows by a computer verification. Thus, we may assume $r\ge 3$. Then $x^{-1}v^{2^{r-2}}x=v^{-2^{r-2}}$. Moreover, $\Omega_2(\Phi(P))\subseteq\mho(\Phi(P))=\Phi(\Phi(P))\subseteq\Z(\Phi(P))$, since $\Phi(P)$ is abelian or minimal nonabelian depending on $\eta$. This shows that $\beta$ acts as inversion on $\Omega_2(\Phi(P))$ and thus cannot fix $\delta$. It follows that $m=1$. Then $\lvert\N_P(Q)\cap\Phi(Q)\rvert\le 8$. In particular there exists an element $\rho\in\Omega_2(\Phi(P))\setminus(\N_P(Q)\cap\Phi(Q))$. Then $^\rho\beta=\beta\rho^{-2}\in Q$. Since $Q=\langle\beta\rangle(Q\cap\Phi(P))$, we derive the contradiction $\rho\in\N_P(Q)$.

Now assume $Q\cong C_{2^m}\times Q_8$ for some $m\ge 1$. 
We choose again $\beta:=xa^{2i}v^j\in Q\setminus\Phi(P)$. 
If $\Phi(P)$ contains a subgroup isomorphic to $Q_8$, then $\Omega_2(\Phi(P))$ cannot be abelian. %However, this contradicts $|(\Phi(P))'|\le 2$. 
So, in case $m=1$ we have $\N_P(Q)\cap\Phi(P)\cong C_8\times C_2$. Then the argument above reveals a contradiction (using $r\ge 3$). Now let $m\ge 2$.
We write $Q=\langle q_1\rangle\times\langle q_2,q_3\rangle$ where $\langle q_1\rangle\cong C_{2^m}$ and $\langle q_2,q_3\rangle\cong Q_8$. In case $q_1\notin\Phi(P)$ we can choose $\beta=q_1$. In any case it follows that $\beta$ fixes an element of order $4$ in $Q\cap\Phi(P)$. This leads to a contradiction as above.

Finally suppose that $Q\cong C_{2^m}\mathop{\ast}Q_8\cong C_{2^m}\mathop{\ast}D_8$ for some $m\ge 2$. Here we can choose $\beta\in Q\setminus\Phi(P)$ as an involution. Then there is always an element of order $4$ in $Q\cap\Phi(P)$ which is fixed by $\beta$. The contradiction follows as before.
\end{proof}

\begin{Proposition}
Let $P$ be a bicyclic $2$-group such that $P'\cong C_{2^r}\times C_{2^{r+s+1}}$ for some $r,s\ge 2$. Then every fusion system on $P$ is nilpotent.
\end{Proposition}
\begin{proof}
Here Theorem~4.13 in \cite{Janko} applies. The proof is a combination of the proofs of Proposition~\ref{nonhomo} and Proposition~\ref{nonhomo2}. In fact for part (a) of Theorem~4.13 we can copy the proof of Proposition~\ref{nonhomo}. Similarly the arguments of Proposition~\ref{nonhomo2} remain correct for case (b). Here observe that there is no need to discuss the case $r=2$ separately, since $x^{-1}v^{2^{r+s-1}}x=v^{-2^{r+s-1}}$.
\end{proof}

Now it suffices to consider the case where $P'$ contains a cyclic maximal subgroup. If $P'$ is noncyclic, Theorem~4.8 in \cite{Janko} applies. This case is more complicated, since $|P/P'|$ is not bounded anymore.

\begin{Proposition}
Let $P$ be a bicyclic $2$-group such that $P'\cong C_{2^n}\times C_2$ for some $n\ge 2$, and $P/\Phi(P')$ has a normal elementary abelian subgroup of order $8$. Then every fusion system on $P$ is nilpotent.
\end{Proposition}
\begin{proof}
%First we exploit the structure of $\Phi(P)=\langle a^2\rangle\langle v\rangle$. 
There are two possibilities for $P$ according to if $\Z(P)$ is cyclic or not. We handle them separately. 

\textbf{Case~1: }$\Z(P)$ noncyclic.\\
Then $a^{2^m}=uz^{\eta}$. Moreover,
\begin{equation}\label{a2v}
a^{-2}va^2=a^{-1}vuv^{2+4s}a=a^{-1}uv^{3+4s}a=u(uv^{3+4s})^{3+4s}=v^{(3+4s)^2}\in v\langle v^8\rangle.
\end{equation}
Using this we see that $\langle a^{2^{m-1}},v^{2^{n-2}}\rangle\cong C_4^2$. Thus, Lemma~\ref{metanormal} implies $\Omega_2(\Phi(P))=\langle a^{2^{m-1}},v^{2^{n-2}}\rangle$. As usual we assume that there is an $\mathcal{F}$-essential subgroup $Q\cong C_{2^t}\times C_2^2$ for some $t\ge 1$. Then $Q\cap\Phi(P)\cong C_{2^t}\times C_2$, since $\Phi(P)$ has rank $2$. For $t=1$ we obtain $Q\cap\Phi(P)=\Omega(\Phi(P))\subseteq\Z(P)$. Write $\overline{P}:=P/\Omega(\Phi(P))$, $\overline{Q}:=Q/\Omega(\Phi(P))$ and so on. Then $\C_{\overline{P}}(\overline{Q})\subseteq\overline{\N_P(Q)}$. So by Satz~III.14.23 in \cite{Huppert}, $\overline{P}$ has maximal class. Hence, $P'=\Phi(P)$ and $m=1$. Contradiction. Thus, we may assume $t\ge 2$. Then as usual we can find an element $\delta\in Q\cap\Phi(P)$ of order $4$ which is fixed by some involution $\beta\in Q\setminus\Phi(P)$. We write $\delta=a^{2^{m-1}d_1}v^{2^{n-2}d_2}$ and $\beta=xv^ja^{2i}$. Assume first that $2\mid d_1$. Then $2\nmid d_2$. Since $a^{2^m}v^{2^{n-2}}\in\Z(\Phi(P))$, it follows that $\delta={^\beta\delta}={^x\delta}=\delta^{-1}$. This contradiction shows $2\nmid d_1$. After replacing $\delta$ with its inverse if necessary, we can assume $d_1=1$. Now we consider $\beta$. We have
\[1=\beta^2=(xv^ja^{2i})^2\equiv x^2v^{2j}a^{4i}\equiv a^{4i}\pmod{P'}.\]
Since 
\[2^{n+m}=\lvert\Phi(P)\rvert=\frac{|\langle a^2\rangle||P'|}{|\langle a^2\rangle\cap P'|}=\frac{2^{n+m+1}}{|\langle a^2\rangle\cap P'|},\]
we get $2^{m-2}\mid i$. In case $i=2^{m-2}$ we get the contradiction 
\[\langle z\rangle\ni x^2=xv^{j-2^{n-2}d_2}xv^{j-2^{n-2}d_2}=(\beta\delta^{-1})^2=\delta^2\in u\langle z\rangle.\]
Hence, $2^{m-1}\mid i$. 
So, after multiplying $\beta$ with $\delta^2$ if necessary, we may assume $i=0$, i.\,e. $\beta=xv^j$. Then $1=xv^jxv^j=x^2$. 
Conjugation with $a^{-1}$ gives $\beta=a^{-1}xv^ja=xv^{-1}a^{-1}v^ja=xu^jv^{(3+4s)j-1}$. Since $u\in Q$, we may assume that $\beta=xv^{2j}$. 
%Since $P/R$ is minimal nonabelian of type $(m,1)$, we can conjugate $Q$ such that $\beta=xu^iv^{2j}$ and $\beta=xv^{2j}$ after changing the representative again.
After we conjugate $Q$ by $v^j$, we even obtain $\beta=x$. Since $x(a^2v^i)x^{-1}=a^2uv^{4(1+s)-i}$, no element of the form $a^2v^i$ is fixed by $x$. On the other hand
\[x(a^4v^i)x^{-1}=(a^2uv^{4(1+s)})^2v^{-i}=a^4v^{4(1+s)(3+4s)^2+4(1+s)-i}.\]
This shows that there is an $i$ such that $a^4v^i=:\lambda$ is fixed by $x$. Assume there is another element $\lambda_1:=a^4v^j$ which is also fixed by $x$. 
Then $\lambda^{-1}\lambda_1=v^{j-i}\in\langle z\rangle$. This holds in a similar way for elements containing higher powers of $a$. In particular $u=a^{2^m}z^{\eta}\in\langle\lambda,z\rangle$. 
Recall that $\Phi(P)=\langle v\rangle\rtimes\langle a^2\rangle$. This shows $\C_{\Phi(P)}(x)=\langle\lambda\rangle\times\langle z\rangle\cong C_{2^{m-1}}\times C_2$. Since $Q\cap\Phi(P)\subseteq\C_{\Phi(P)}(x)$ and $Q=(Q\cap\Phi(P))\langle x\rangle$, we deduce $\C_{\Phi(P)}(x)\subseteq\C_P(Q)\subseteq Q$. Moreover, $Q\cap\Phi(P)=\C_{\Phi(P)}(x)$ and $t=m-1$. Therefore, $Q=\langle\lambda,x,z\rangle$. The calculation above shows that there is an element $\mu:=a^2v^j$ such that $^\mu x=ux\in Q$. Now $\mu^2\in\C_{\Phi(P)}(x)$ implies $\C_{\Phi(P)}(x)=\langle\mu^2,z\rangle$ and $\mu\in\N_P(Q)=Q\langle v^{2^{n-2}}\rangle$. Contradiction.

Now assume $Q\cong C_{2^t}\times Q_8$ for some $t\ge 1$. 
Since $\Phi(P)$ does not contain a subgroup isomorphic to $Q_8$, we see that $\Omega(\Phi(P))\subseteq Q$. First assume $t=1$. Then we look again at the quotients $\overline{P}:=P/\Omega(\Phi(P))$ and $\overline{Q}:=Q/\Omega(\Phi(P))\cong C_2^2$. Since $\N_P(Q)$ acts nontrivially on $\overline{Q}$, we get $\C_{\overline{P}}(\overline{Q})\subseteq\overline{Q}$. In particular Proposition~1.8 in \cite{Berkovich1} implies that $\overline{P}$ has maximal class. This leads to a contradiction as in the first part of the proof. Thus, we may assume $t\ge 2$ from now on. Then $\Omega_2(\Phi(P))\subseteq Q$. Since $Q$ contains more elements of order $4$ than $\Phi(P)$, we can choose $\beta\in Q\setminus\Phi(P)$ of order $4$. Write $\beta=xa^{2i}v^j$. Then $\beta^2\in\Omega(\Phi(P))\subseteq P'$. So the same discussion as above shows that we can assume $\beta=x$. In particular $|\langle x\rangle|=4$. 
%Since $\C_Q(x)\cong C_{2^t}\times C_4$, we have $Q=\C_Q(x)\langle v^{2^{n-2}}\rangle$. This 
Since $\C_{\Phi(P)}(x)$ is abelian, $\lambda$ centralizes $(\C_Q(x)\cap\Phi(P))\langle x\rangle\langle v^{2^{n-2}}\rangle=\C_Q(x)\langle v^{2^{n-2}}\rangle=Q$. This shows $\lambda\in Q$ and $t=m-1$ again.
More precisely we have $Q=\langle\lambda\rangle\times\langle v^{2^{n-2}},x\rangle$. Equation~\eqref{a2v} shows that $v^{2^{n-3}}$ still lies in the center of $\Phi(P)$. It follows easily that $\N_P(Q)=Q\langle v^{2^{n-3}}\rangle$. However, as above we also have $\mu\in\N_P(Q)$. Contradiction.

Finally, the case $Q\cong C_{2^t}\mathop{\ast}Q_8$ cannot occur, since $\Z(P)$ is noncyclic.

\textbf{Case~2: }$\Z(P)$ cyclic.\\
Here we have $a^{2^m}=uv^{2^{n-2}}z^{\eta}$, $n\ge m+2\ge 4$ and $1+s\not\equiv 0\pmod{2^{n-3}}$. Again we begin with $Q\cong C_{2^t}\times C_2^2$ for some $t\ge 1$. By Theorem~4.3(b) in \cite{Janko} we still have $\langle u,z\rangle=\Omega(\Z(\Phi(P)))$. Since $\Phi(P)$ does not have maximal class, also $\langle u,z\rangle=\Omega(\Phi(P))$ holds. In particular $\Omega(\Phi(P))\subseteq Q$. In case $t=1$ we see that $P/\Omega(\Phi(P))$ has maximal class which leads to a contradiction as before. Thus, $t\ge 2$. Since $u\in\Z(\Phi(P))$, Equation~\eqref{a2v} is still true. 
Hence, $\Omega_2(\Phi(P))=\langle a^{2^{m-1}}v^{2^{n-3}},v^{2^{n-2}}\rangle\cong C_4^2$.
%Hence, $\Phi(P)'\subseteq\langle v^8\rangle$. This shows that $\overline{\Phi(P)}:=\Phi(P)/\Omega(\Phi(P))$ does not have maximal class. Hence, $\Omega(\overline{\Phi(P)})\cong C_2^2$ (see Exercise~1.85 in \cite{Berkovich1}) and $\lvert\Omega_2(\Phi(P))\rvert=16$. Since $\langle u,v^{2^{n-2}}\rangle\subseteq\Z(\Phi(P))\cap\Omega_2(\Phi(P))\subseteq\Z(\Omega_2(\Phi(P)))$, it follows that $\Omega_2(\Phi(P))\cong C_4^2$. 
We choose an involution $\beta=xv^ja^{2i}\in Q\setminus\Phi(P)$. Then as usual $v^{2^{n-2}}\in\N_P(Q)\setminus Q$. Since $a^{2^m}\in\langle u\rangle\times\langle v^{2^{n-2}}\rangle$, we find an element $\delta=a^{2^{m-1}}v^{d_1}\in Q\cap \Omega_2(\Phi(P))$ of order $4$ fixed by $\beta$. Now exactly the same argument as in Case~1 shows that $\beta=x$ after changing the representative of $\beta$ and conjugation of $Q$ if necessary. Similarly we get $\lambda:=a^4v^j\in\C_{\Phi(P)}(x)$. Moreover, $u=a^{2^m}v^{-2^{n-2}}z^{\eta}\in\{\lambda^{2^{m-2}},\lambda^{2^{m-2}}z\}$. Therefore, $\C_{\Phi(P)}(x)=\langle\lambda\rangle\times\langle z\rangle\cong C_{2^{m-1}}\times C_2$. The contradiction follows as before.

Now assume that $Q\cong C_{2^t}\times Q_8$ or $Q\cong C_{2^{t+1}}\mathop{\ast} Q_8$ for some $t\ge 1$.
Proposition~\ref{centerfree} shows that $\mathcal{F}=\C_{\mathcal{F}}(\langle z\rangle)$. Theorem~6.3 in \cite{Linckelmann} implies that $\overline{Q}:=Q/\langle z\rangle$
%\cong C_{2^t}\times C_2^2$ 
is an $\mathcal{F}/\langle z\rangle$-essential subgroup of $\overline{P}:=P/\langle z\rangle$. Now $\overline{P}$ is bicyclic and has commutator subgroup isomorphic to $C_{2^{n-1}}\times C_2$. Hence the result follows by induction on $t$.
\end{proof}

Combining these propositions we deduce one of the main results of this paper. 

\begin{Theorem}\label{cycliccom}
Every fusion system on a bicyclic $2$-group $P$ is nilpotent unless $P'$ is cyclic.
\end{Theorem}

It seems that there is no general reason for Theorem~\ref{cycliccom}. For example there are nonnilpotent fusion systems on $2$-groups of rank $2$ with noncyclic commutator subgroup.

For the convenience of the reader we state a consequence for finite groups.

\begin{Corollary}
Let $G$ be a finite group with bicyclic Sylow $2$-subgroup $P$. If $P'$ is noncyclic, then $P$ has a normal complement in $G$.
\end{Corollary}

\subsection{The case $P'$ cyclic}

In this section we consider the remaining case where the bicyclic $2$-group $P$ has cyclic commutator subgroup. Here Theorem~4.4 in \cite{Janko} plays an important role. The following theorem classifies all fusion systems on bicyclic $2$-groups together with some more information.

\begin{Theorem}\label{main}
Let $\mathcal{F}$ be a fusion system on a bicyclic $2$-group $P$. Then one of the following holds:
\begin{enumerate}[(1)]
\item\label{nilcase} $\mathcal{F}$ is nilpotent, i.\,e. $\mathcal{F}=\mathcal{F}_P(P)$.
\item $P\cong C_{2^n}^2$ and $\mathcal{F}=\mathcal{F}_P(P\rtimes C_3)$ for some $n\ge 1$.
\item $P\cong D_{2^n}$ for some $n\ge 3$ and $\mathcal{F}=\mathcal{F}_P(\PGL(2,5^{2^{n-3}}))$ or $\mathcal{F}=\mathcal{F}_P(\PSL(2,5^{2^{n-2}}))$. % where $q$ is a suitable odd prime power. % and in the latter case $q\equiv\pm1\pmod{8}$.
Moreover, $\mathcal{F}$ provides one respectively two essential subgroups isomorphic to $C_2^2$ up to conjugation. 
\item $P\cong Q_8$ and $\mathcal{F}=\mathcal{F}_P(\SL(2,3))$ is controlled, i.\,e. there are no $\mathcal{F}$-essential subgroups.
\item $P\cong Q_{2^n}$ for some $n\ge 4$ and $\mathcal{F}=\mathcal{F}_P(\SL(2,5^{2^{n-4}}).C_2)$ or $\mathcal{F}=\mathcal{F}_P(\SL(2,5^{2^{n-3}}))$. Moreover, $\mathcal{F}$ provides one respectively two essential subgroups isomorphic to $Q_8$ up to conjugation. 
\item $P\cong SD_{2^n}$ for some $n\ge 4$ and $\mathcal{F}=\mathcal{F}_P(\PSL(2,5^{2^{n-3}})\rtimes C_2)$, $\mathcal{F}=\mathcal{F}_P(\GL(2,q))$ or $\mathcal{F}=\mathcal{F}_P(\PSL(3,q))$ where in the last two cases $q$ is a suitable prime power such that $q\equiv 3\pmod{4}$. Moreover, in the first (resp. second) case $C_2^2$ (resp. $Q_8$) is the only $\mathcal{F}$-essential subgroup up to conjugation, in the last case both are $\mathcal{F}$-essential and these are the only ones up to conjugation.
\item $P\cong C_{2^n}\wr C_2$ for some $n\ge 2$ and $\mathcal{F}=\mathcal{F}_P(C_{2^n}^2\rtimes S_3)$, $\mathcal{F}=\mathcal{F}_P(\GL(2,q))$ or $\mathcal{F}=\mathcal{F}_P(\PSL(3,q))$ where in the last two cases $q\equiv 1\pmod{4}$. Moreover, in the first (resp. second) case $C_{2^n}^2$ (resp. $C_{2^n}\mathop{\ast}Q_8$) is the only $\mathcal{F}$-essential subgroup up to conjugation, in the last case both are $\mathcal{F}$-essential and these are the only ones up to conjugation.
\item $P\cong C_2^2\rtimes C_{2^n}$ is minimal nonabelian of type $(n,1)$ for some $n\ge 2$ and $\mathcal{F}=\mathcal{F}_P(A_4\rtimes C_{2^n})$.
% where $C_{2^n}$ acts as a transposition in $\Aut(A_4)\cong S_4$. 
Moreover, $C_{2^{n-1}}\times C_2^2$ is the only $\mathcal{F}$-essential subgroup of $P$.
\item\label{dnonsp} $P\cong\langle v,x,a\mid v^{2^n}=x^2=1,\ ^xv=v^{-1},\ a^{2^m}=v^{2^{n-1}},\ ^av=v^{-1+2^{n-m+1}},\ ^ax=vx\rangle\cong D_{2^{n+1}}.C_{2^m}$ for $n>m>1$ and $\mathcal{F}=\mathcal{F}_P(\PSL(2,5^{2^{n-1}}).C_{2^m})$. Moreover, $C_{2^{m-1}}\times C_2^2$ is the only $\mathcal{F}$-essential subgroup up to conjugation.
\item\label{dsp} $P\cong \langle v,x,a\mid v^{2^n}=x^2=a^{2^m}=1,\ ^xv=v^{-1},\ ^av=v^{-1+2^i},\ ^ax=vx\rangle\cong D_{2^{n+1}}\rtimes C_{2^m}$ for $\max(2,n-m+2)\le i\le n$ and $n,m\ge 2$. Moreover, $\mathcal{F}=\mathcal{F}_P(\PSL(2,5^{2^{n-1}})\rtimes C_{2^m})$ and $C_{2^{m-1}}\times C_2^2$ is the only $\mathcal{F}$-essential subgroup up to conjugation. In case $i=n$ there are two possibilities for $\mathcal{F}$ which differ by $\Z(\mathcal{F})\in\{\langle a^2\rangle,\langle a^2v^{2^{n-1}}\rangle\}$.
\item\label{q8} $P\cong\langle v,x,a\mid v^{2^n}=1,\ x^2=a^{2^m}=v^{2^{n-1}},\ ^xv=v^{-1},\ ^av=v^{-1+2^{n-m+1}},\ ^ax=vx\rangle\cong Q_{2^{n+1}}.C_{2^m}$ for $n>m>1$ and $\mathcal{F}=\mathcal{F}_P(\SL(2,5^{2^{n-2}}).C_{2^m})$. Moreover, $C_{2^{m-1}}\times Q_8$ is the only $\mathcal{F}$-essential subgroup up to conjugation.
\item $P\cong \langle v,x,a\mid v^{2^n}=a^{2^m}=1,\ x^2=v^{2^{n-1}},\ ^xv=v^{-1},\ ^av=v^{-1+2^i},\ ^ax=vx\rangle\cong Q_{2^{n+1}}\rtimes C_{2^m}$ for $\max(2,n-m+2)\le i\le n$ and $n,m\ge 2$. Moreover, $\mathcal{F}=\mathcal{F}_P(\SL(2,5^{2^{n-2}})\rtimes C_{2^m})$ and $C_{2^{m-1}}\times Q_8$ is the only $\mathcal{F}$-essential subgroup up to conjugation.
\item $P\cong \langle v,x,a\mid v^{2^n}=a^{2^m}=1,\ x^2=v^{2^{n-1}},\ ^xv=v^{-1},\ ^av=v^{-1+2^{n-m+1}},\ ^ax=vx\rangle\cong Q_{2^{n+1}}\rtimes C_{2^m}$ for $n>m>1$ and $\mathcal{F}=\mathcal{F}_P(\SL(2,5^{2^{n-2}})\rtimes C_{2^m})$. Moreover, $C_{2^m}\mathop{\ast}Q_8$ is the only $\mathcal{F}$-essential subgroup up to conjugation.
\item\label{last} $P\cong\langle v,x,a\mid v^{2^n}=1,\ x^2=a^{2^m}=v^{2^{n-1}},\ ^xv=v^{-1},\ ^av=v^{-1+2^i},\ ^ax=vx\rangle\cong Q_{2^{n+1}}.C_{2^m}$ for $\max(2,n-m+2)\le i\le n$ and $n,m\ge 2$. In case $m=n$, we have $i\ne n$. Moreover, $\mathcal{F}=\mathcal{F}_P(\SL(2,5^{2^{n-2}}).C_{2^m})$ and $C_{2^m}\mathop{\ast}Q_8$ is the only $\mathcal{F}$-essential subgroup up to conjugation.
\end{enumerate}
In particular $\mathcal{F}$ is nonexotic. Conversely, for every group described in these cases there exists a fusion system with the given properties. Moreover, different parameters give nonisomorphic groups.
\end{Theorem}
\begin{proof}
Assume that $\mathcal{F}$ is nonnilpotent. By Theorem~\ref{cycliccom}, $P'$ is cyclic.
The case $P\cong Q_8$ is easy. For the other metacyclic cases and the case $P\cong C_{2^n}\wr C_2$ we refer to Theorem~5.3 in \cite{CravenGlesser}. 
Here we add a few additional information. An induction on $i\ge 2$ shows that $5^{2^{i-2}}\equiv 1+2^i\pmod{2^{i+1}}$. This implies that the Sylow $2$-subgroups of $\SL(2,5^{2^{n-3}})$, $\PSL(2,5^{2^{n-2}})$ and so on have the right order. For the groups $SD_{2^n}$ and $C_{2^n}\wr C_2$ it is a priori not clear if for every $n$ an odd prime power $q$ can be found. However, this can be shown using Dirichlet's Prime Number Theorem (compare with Theorem~6.2 in \cite{pnilpotentforcing}). Hence, for a given $n$ all these fusion systems can be constructed.

Using Proposition~\ref{rank2ess} we can assume that every $\mathcal{F}$-essential subgroup has rank $3$. Finally by Proposition~\ref{E8normal} it remains to consider $|P'|>2$. Hence, let $P$ be as in Theorem~4.4 in \cite{Janko}. 
We adapt our notation slightly as follows. We replace $a$ by $a^{-1}$ in order to write $^av$ instead of $v^a$. Then we have $^ax=vx$. After replacing $v$ by a suitable power, we may assume that $i$ is a $2$-power (accordingly we need to change $x$ to $v^\eta x$ for a suitable number $\eta$). Then we can also replace $i$ by $2+\log i$. This gives
\begin{equation}\label{pres}
P\cong\langle v,x,a\mid v^{2^n}=1,\ x^2,a^{2^m}\in\langle v^{2^{n-1}}\rangle,\ ^xv=v^{-1},\ ^av=v^{-1+2^i},\ ^ax=vx\rangle.
\end{equation}
Since Theorem~4.4 in \cite{Janko} also states that $v$ and $a^{2^{m-1}}$ commute, we obtain $i\in\{\max(n-m+1,2),\ldots,n\}$. We set $z:=v^{2^{n-1}}$ as in \cite{Janko}. Moreover, let $\lambda:=v^{-2^{i-1}}a^2$. Then
\[x\lambda x^{-1}=v^{2^{i-1}}(v^{-1}a)^2=v^{-2^{i-1}}a^2=\lambda\]
and $\lambda\in\C_{\Phi(P)}(x)$. Assume that also $v^ja^2\in\C_{\Phi(P)}(x)$. Then we get $v^ja^2\in\{\lambda,\lambda z\}$. Hence, $\C_{\Phi(P)}(x)\in\{\langle\lambda\rangle, \langle\lambda\rangle\times\langle z\rangle\}$.
It should be pointed out that it was not shown in \cite{Janko} that these presentations really give groups of order $2^{n+m+1}$ (although some evidence by computer results is stated). However, we assume in the first part of the proof that these groups with the “right” order exist. Later we construct $\mathcal{F}$ as a fusion of a finite group and it will be clear that $P$ shows up as a Sylow $2$-subgroup of order $2^{n+m+1}$.
Now we distinguish between the three different types of essential subgroups.

\textbf{Case~1:} $Q\cong C_{2^t}\times C_2^2$ is $\mathcal{F}$-essential in $P$ for some $t\ge 1$.\\
As usual $Q\le M=E\langle a^2\rangle$. Since $Q\cap E$ is abelian and $Q/Q\cap E\cong QE/E\le P/E$ is cyclic, it follows that $E$ is dihedral and $Q\cap E\cong C_2^2$. After conjugation of $Q$ we may assume $Q\cap E\in\{\langle z,x\rangle,\langle z,vx\rangle\}$. Further conjugation with $a$ gives $Q\cap E=\langle z,x\rangle$. 
Since $\C_Q(x)\cap\Phi(P)$ is noncyclic, it follows that $\C_{\Phi(P)}(x)=\langle\lambda\rangle\times\langle z\rangle\cong C_{2^{m-1}}\times C_2$. As usual we obtain $Q=\langle\lambda,z,x\rangle$ and $t=m-1$. Moreover, $a^2va^{-2}\equiv v\pmod{\langle v^8\rangle}$. Hence, $\N_P(Q)=\langle\lambda,v^{2^{n-2}},x\rangle$.  

We prove that $Q$ is the only $\mathcal{F}$-essential subgroup of $P$ up to conjugation. 
If there is an $\mathcal{F}$-essential subgroup of rank $2$, then Proposition~\ref{rank2ess} implies that $P$ is a wreath product. However, by the proof of Theorem~5.3 in \cite{CravenGlesser} all the other $\mathcal{F}$-essential subgroups are of type $C_{2^r}\mathop{\ast}Q_8$. Hence, this case cannot occur. Thus, by construction it is clear that $Q$ is the only abelian $\mathcal{F}$-essential subgroup up to conjugation. Now assume $Q_1\cong C_{2^s}\times Q_8$ is also $\mathcal{F}$-essential. Since $Q_1$ has three involutions, $Q_1\cap E$ is cyclic or isomorphic to $C_2^2$. In either case $Q/Q\cap Q\cong QE/E\le P/E$ cannot be cyclic. Contradiction. Suppose now that $Q_1\cong C_{2^s}\mathop{\ast}Q_8\cong C_{2^s}\mathop{\ast}D_8$ for some $s\ge 2$. Then $Q_1\cap E$ cannot be cyclic, since $Q_1$ has rank $3$. Suppose $Q_1\cap E\cong C_2^2$. Then $\Omega(\Z(Q_1))\subseteq Q_1\cap E$ and $\exp Q_1/Q_1\cap E\le 2^{s-1}$. On the other hand $|Q_1/Q_1\cap E|=2^s$. In particular $Q_1/Q_1\cap E\cong Q_1E/E\le P/E$ cannot be cyclic. It follows that $Q_1\cap E$ must be a (nonabelian) dihedral group. Hence, $2^{s-1}|Q_1\cap E|=|(Q_1\cap E)\Z(Q_1)|\le|Q_1|=2^{s+2}$ and $Q_1\cap E\cong D_8$. After conjugation of $Q_1$ we have $Q_1\cap E=\langle v^{2^{n-2}},x\rangle$. Let $\lambda_1\in \Z(Q_1)\setminus E$ be an element of order $2^s$ such that $\lambda_1^{2^{s-1}}=z$. Since $x\in Q_1$, we have $\lambda_1^2\in\C_{\Phi(P)}(x)=\langle\lambda\rangle\times\langle z\rangle$. This implies $s=2$ and $\lambda_1\notin\Phi(P)$. Since $Q_1=(Q_1\cap\Phi(P))\langle x\rangle$, we obtain $\lambda_1x\in\C_{\Phi(P)}(x)$. But this contradicts $z=\lambda_1^2=(\lambda_1x)^2$. Hence, we have proved that $Q$ is in fact the only $\mathcal{F}$-essential subgroup of $P$ up to conjugation. 

Now we try to pin down the structure of $P$ more precisely. 
We show by induction on $j\ge 0$ that $\lambda^{2^j}=v^{2^{i+j-1}\nu}a^{2^{j+1}}$ for an odd number $\nu$. This is clear for $j=0$. For arbitrary $j\ge 1$ we have
\[\lambda^{2^j}=\lambda^{2^{j-1}}\lambda^{2^{j-1}}=v^{2^{i+j-2}\nu}a^{2^j}v^{2^{i+j-2}\nu}a^{2^j}=v^{2^{i+j-2}\nu(-1+2^i)^{2^j}+2^{i+j-2}\nu}a^{2^{j+1}}=v^{2^{i+j-2}\nu((-1+2^i)^{2^j}+1)}a^{2^{j+1}},\]
and the claim follows. In particular we obtain 
\begin{equation}\label{congi}
1=\lambda^{2^{m-1}}=v^{2^{i+m-2}\nu}a^{2^m}.
\end{equation}
We distinguish whether $P$ splits or not.

\textbf{Case~1(a):} $a^{2^m}=z$.\\
Here Equation~\eqref{congi} shows $i=n-m+1$. Then $n>m>1$, and the isomorphism type of $P$ is completely determined by $m$ and $n$. We show next that $\mathcal{F}$ is uniquely determined. For this we need to describe the action of $\Aut_{\mathcal{F}}(Q)$ in order to apply Alperin's Fusion Theorem. As in the proof of Proposition~\ref{E8normal}, $\Aut_{\mathcal{F}}(Q)$ acts on $\langle x,z\rangle$ or on $\langle x\lambda^{2^{m-2}},z\rangle$ nontrivially (recall $\N_P(Q)\cong D_8\times C_{2^{m-1}}$). Set $\widetilde{x}:=x\lambda^{2^{m-2}}$ and $\widetilde{a}:=av^{2^{n-2}}$. Then as above $\widetilde{x}=xv^{\pm2^{n-2}}a^{2^{m-1}}$. Hence, $\widetilde{x}^2=1$ and $^{\widetilde{x}}v=v^{-1}$. Moreover, $\widetilde{a}^2=a^2$ and thus $\widetilde{a}^{2^m}=z$. Finally, $^{\widetilde{a}}v={^av}$ and $^{\widetilde{a}}\widetilde{x}={^a(xzv^{\pm2^{n-2}}a^{2^{m-1}})}=vxzv^{\mp2^{n-2}}a^{2^{m-1}}=v\widetilde{x}$. Hence, $v$, $\widetilde{x}$ and $\widetilde{a}$ satisfy the same relations as $v$, $x$ and $a$. Obviously, $P=\langle v,\widetilde{x},\widetilde{a}\rangle$. Therefore, we may replace $x$ by $\widetilde{x}$ and $a$ by $\widetilde{a}$. After doing this if necessary, we see that $\Aut_{\mathcal{F}}(Q)$ acts nontrivially on $\langle x,z\rangle$ (observe that $Q$ remains fixed under this transformation). As usual it follows that $\C_Q(\Aut_{\mathcal{F}}(Q))\in\{\langle\lambda\rangle,\langle\lambda z\rangle\}$ (compare with proof of Proposition~\ref{E8normal}). Define $\widetilde{a}:=a^{1+2^{m-1}}$ and $\widetilde{v}:=v^{1+2^{n-1}}=vz$. Then $\widetilde{a}^2=a^2z$, $\widetilde{a}^{2^m}=z$, $\widetilde{v}^{2^n}=1$, $^x\widetilde{v}=\widetilde{v}^{-1}$ and $^{\widetilde{a}}\widetilde{v}=\widetilde{v}^{-1+2^{n-m+1}}$. Now we show by induction on $j\ge 1$ that $a^{2^j}xa^{-2^j}=v^{2^{n-m+j}\nu}x$ for an odd integer $\nu$. For $j=1$ we have $a^2xa^{-2}={^a(vx)}=v^{2^{n-m+1}}x$. For arbitrary $j\ge 1$ induction gives
\[a^{2^{j+1}}xa^{-2^{j+1}}=a^{2^j}(a^{2^j}xa^{-2^j})a^{-2^j}=a^{2^j}(v^{2^{n-m+j}\nu}x)a^{-2^j}=v^{2^{n-m+j}\nu((-1+2^{n-m+1})^{2^j}+1)}x,\]
and the claim follows. In particular $a^{2^{m-1}}xa^{-2^{m-1}}=zvx$ and $^{\widetilde{a}}x=\widetilde{v}x$. Obviously, $P=\langle\widetilde{v},\widetilde{a},x\rangle$. Hence, we may replace $v$, $a$, $x$ by $\widetilde{v}$, $\widetilde{a}$, $x$ if necessary. Under this transformation $Q$ and $\langle x,z\rangle$ remain fixed as sets and $\lambda$ goes to $\lambda z$. So, we may assume $\C_Q(\Aut_{\mathcal{F}}(Q))=\langle\lambda\rangle$. Then the action on $\Aut_{\mathcal{F}}(Q)$ on $Q$ is completely described. In particular $\mathcal{F}$ is uniquely determined.

It remains to prove that $P$ and $\mathcal{F}$ really exist. Let $q:=5^{2^{n-1}}$.
%Choose an odd prime $p$ such that $p^m\equiv\pm1\pmod{2^{n+1}}$.
It is not hard to verify that $H:=\PSL(2,q)$ has Sylow $2$-subgroup $E\cong D_{2^{n+1}}$. More precisely, $E$ can be generated by the following matrices
\begin{align*}
v&:=\begin{pmatrix}\omega&0\\0&\omega^{-1}\end{pmatrix},&x&:=\begin{pmatrix}
0&1\\-1&0\end{pmatrix}
\end{align*}
where $\omega\in\mathbb{F}_q^\times$ has order $2^{n+1}$. Moreover, we regard these matrices modulo $\Z(\SL(2,q))=\langle -1_2\rangle$. Now consider the matrix $a_1:=\bigl(\begin{smallmatrix}0&\omega\\-1&0\end{smallmatrix}\bigr)\in\GL(2,q)/\Z(\SL(2,q))$. Then $a_1$ acts on $H$ and a calculation shows $^{a_1}v=v^{-1}$ and $^{a_1}x=vx$. 
Let $\gamma_1$ be the Frobenius automorphism of $\mathbb{F}_q$ with respect to $\mathbb{F}_5$, i.\,e. $\gamma_1(\tau)=\tau^5$ for $\tau\in\mathbb{F}_q$. 
As usual we may regard $\gamma_1$ as an automorphism of $H$. Let $\gamma:=\gamma_1^{2^{n-m-1}}$ so that $|\langle\gamma\rangle|=2^m$. 
%We set $\gamma:=\gamma_1^{2^{n-m-1}}$. Then $|\langle\gamma\rangle|=2^m$.
Recall that $(\mathbb{Z}/2^{n+1}\mathbb{Z})^\times=\langle 5+2^{n+1}\mathbb{Z}\rangle\times\langle -1+2^{n+1}\mathbb{Z}\rangle\cong C_{2^{n-1}}\times C_2$. 
%In particular $(\mathbb{Z}/2^n\mathbb{Z})^\times$ contains exactly two subgroups isomorphic to $C_{2^m}$ (notice $n>m>1$). 
It is easy to show that $\langle 5^{2^{n-m-1}}+2^{n+1}\mathbb{Z}\rangle$ and $\langle 1-2^{n-m+1}+2^{n+1}\mathbb{Z}\rangle$ are subgroups of $(\mathbb{Z}/2^{n+1}\mathbb{Z})^\times$ of order $2^m$.
Since 
\[5^{2^{n-m-1}}\equiv 1-2^{n-m+1}\pmod{8},\]
it follows that 
\[\langle 5^{2^{n-m-1}}+2^{n+1}\mathbb{Z}\rangle=\langle 1-2^{n-m+1}+2^{n+1}\mathbb{Z}\rangle.\]
In particular we can find an odd integer $\nu$ such that $5^{2^{n-m-1}\nu}\equiv 1-2^{n-m+1}\pmod{2^{n+1}}$. Now we set 
\[a:=a_1\gamma^{\nu}.\]
Since $\gamma_1$ fixes $x$, we obtain $^av=v^{-1+2^{n-m+1}}$ and $^ax=vx$. It remains to show that $a^{2^m}=v^{2^{n-1}}=:z$. Here we identify elements of $H$ with the corresponding inner automorphisms in $\Inn(H)\cong H$.
%It is known that $a_1$ and $\gamma_1$ commute modulo $\Inn(H)\cong H$ (see \cite{Wilson}). Hence $a^{2^m}\in H$ by a slight abuse of notation. 
For an element $u\in H$ we have
\[
a^2(u)=(a_1\gamma^\nu a_1\gamma^\nu)(u)=(a_1\gamma^\nu(a_1))\gamma^{2\nu}(u)(a_1\gamma^\nu(a_1))^{-1}=\left(\begin{pmatrix}
\omega&0\\0&\omega^{5^{2^{n-m-1}\nu}}
\end{pmatrix}\gamma^{2\nu}\right)(u).\]
After multiplying the matrix in the last equation by $\bigl(\begin{smallmatrix}
\omega&0\\0&\omega
\end{smallmatrix}\bigr)^h\in\Z(\GL(2,q))$ for $h:=-(5^{2^{n-m-1}\nu}+1)/2$, we obtain
\[a^2(u)=\left(\begin{pmatrix}
\omega^{2^{n-m}}&0\\0&\omega^{-2^{n-m}}
\end{pmatrix}\gamma^{2\nu}\right)(u),\]
since $(1-5^{2^{n-m-1}\nu})/2\equiv 2^{n-m}\pmod{2^n}$.
Using induction and the same argument we get
\[a^{2^j}=\begin{pmatrix}
\omega^{h_j}&0\\0&\omega^{-h_j}
\end{pmatrix}\gamma^{2^j\nu}\]
where $2^{n-m+j-1}\mid h_j$ and $2^{n-m+j}\nmid h_j$ for $j\ge 1$. In particular $a^{2^m}=z$ as claimed. Now Theorem~15.3.1 in \cite{Hall} shows that the following nonsplit extension exists
\[G:=H\langle a\rangle\cong\PSL(2,5^{2^{n-1}}).C_{2^m}.\]
Moreover, the construction shows that $G$ has Sylow $2$-subgroup $P$. Since $H$ is nonabelian simple, $\mathcal{F}_P(G)$ is nonnilpotent. Hence, $\mathcal{F}=\mathcal{F}_P(G)$.

\textbf{Case~1(b):} $a^{2^m}=1$.\\
Here $P\cong D_{2^{n+1}}\rtimes C_{2^m}$. Moreover, by Equation~\eqref{congi} we have $n-m+2\le i$. 
As in Case~1(a) we may assume that $\Aut_{\mathcal{F}}(Q)$ acts on $\langle x,z\rangle$ using the following automorphism of $P$ if necessary:
% $\mathcal{F}$ is uniquely determined by using the following two automorphisms of $P$:
\[v\mapsto v,\ x\mapsto x\lambda^{2^{m-2}},\ a\mapsto av^{2^{n-2}}.\]
Now assume $i<n$ (and thus $m,n\ge 3$). Here we consider the following map
%If $i<n$ (and thus $m,n\ge 3$), then we can even get $\C_Q(\Aut_{\mathcal{F}}(Q))=\langle\lambda\rangle$ after applying the automorphism
\[v\mapsto v^{1+2^{n-1}}=vz=:\widetilde{v},\ x\mapsto x,\ a\mapsto a^{1+2^{n-i}}=:\widetilde{a}.\]
It can be seen that $\widetilde{v}$, $x$ and $\widetilde{a}$ generate $P$ and satisfy the same relations as $v$, $x$ and $a$. Moreover, as above we have $\lambda^{2^{n-i}}=za^{2^{n-i+1}}$. This shows
\[\lambda\mapsto \widetilde{v}^{-2^{i-1}}\widetilde{a}^2=v^{-2^{i-1}}a^{2+2^{n-i+1}}=\lambda^{1+2^{n-i}}z=(\lambda z)^{1+2^{n-i}}.\]
Hence, we obtain $\C_Q(\Aut_{\mathcal{F}}(Q))=\langle\lambda\rangle$ after applying this automorphism if necessary. This determines $\mathcal{F}$ completely, and we will construct $\mathcal{F}$ later.

We continue by looking at the case $i=n$.
Here we show that $\lambda=za^2$ is not a square in $P$. Assume the contrary, i.\,e. $za^2=(v^jx^ka^l)^2$ for some $j,k,l\in\mathbb{Z}$. Of course $l$ must be odd. In case $k=0$ we get the contradiction $(v^ja^l)^2=a^{2l}$. Thus, $k=1$. Then $[v,xa^l]=1$ and $(v^jxa^l)^2=v^{2j}(xaxa^{-1})a^{2l}=v^{2j-1}a^{2l}$. Again a contradiction. Hence, $\lambda$ is in fact a nonsquare. However, $\lambda z=a^2$ is a square and so is every power. As a consequence, it turns out that the two possibilities $\C_Q(\Aut_{\mathcal{F}}(Q))=\Z(\mathcal{F})=\langle\lambda\rangle$ or $\C_Q(\Aut_{\mathcal{F}}(Q))=\Z(\mathcal{F})=\langle a^2\rangle$ give in fact \emph{nonisomorphic} fusion systems (in the sense of Definition~\ref{equal}). We denote the latter possibility by $\mathcal{F}'$, i.\,e. $\Z(\mathcal{F}')=\langle a^2\rangle$. 

Now for every $i\in\{\max(2,n-m+2),\ldots,n\}$ we construct $P$ and $\mathcal{F}$. After that we explain how to obtain $\mathcal{F}'$ for $i=n$. This works quite similar as in Case~1(a). Let $q$, $H$, $v$, $x$, $a_1$ and $\gamma_1$ as there. It is easy to see that $\langle 1-2^i+2^{n+1}\mathbb{Z}\rangle$ has order $2^{n+1-i}$ as a subgroup of $(\mathbb{Z}/2^{n+1}\mathbb{Z})^\times$. Set $\gamma:=\gamma_1^{2^{i-2}}$. Then $\gamma^{2^m}=1$, since $m+i-2\ge n$. Again we can find an odd integer $\nu$ such that
$5^{2^{i-2}\nu}\equiv 1-2^i\pmod{2^{n+1}}$. Setting $a:=a_1\gamma^{\nu}\in\Aut(H)$ we get $^av=v^{-1+2^i}$ and $^ax=vx$. It remains to prove $a^{2^m}=1$. As above we obtain
\[a^2=\begin{pmatrix}
\omega^{2^{i-1}}&0\\0&\omega^{-2^{i-1}}
\end{pmatrix}\gamma^{2\nu}.\]
This leads to $a^{2^m}=1$. Now we can define $G:=H\rtimes\langle a\rangle$ (notice that the action of $\langle a\rangle$ on $H$ is usually not faithful). It is easy to see that in fact $P\in\Syl_2(G)$ and $\mathcal{F}_P(G)$ is nonnilpotent. Hence, for $i<n$ we get $\mathcal{F}=\mathcal{F}_P(G)$ immediately. Now assume $i=n$. Since $\omega^{2^n}=-1\in\mathbb{F}_q$, we can choose $\omega$ such that $\omega^{2^{n-1}}=2\in\mathbb{F}_5\subseteq\mathbb{F}_q$. Define
\[\alpha:=\begin{pmatrix}3&1\\2&1\end{pmatrix}\in H.\]
%(recall that $\omega^{2^n}=-1\in\mathbb{F}_q$ and $\operatorname{char}\mathbb{F}_q=5$). 
A calculation shows that $\alpha$ has order $3$ and acts on $\langle x,z\rangle$ nontrivially. Moreover, $\gamma^{2\nu}=1$ and $a^2$ is the inner automorphism induced by $z$. In particular $a^2$ does not fix $\alpha$. We can view $\alpha$ as an element of $\Aut_{\mathcal{F}_P(G)}(Q)$. Then $\C_Q(\Aut_{\mathcal{F}_P(G)}(Q))=\langle\lambda\rangle=\Z(\mathcal{F})$ is generated by a nonsquare in $P$. This shows again $\mathcal{F}=\mathcal{F}_P(G)$. It remains to construct $\mathcal{F}'$. Observe that $\gamma$ acts trivially on $\langle v,x\rangle$, since $5^{2^{n-2}}\equiv 1\pmod{2^n}$. Hence, we can replace the automorphism $a$ just by $a_1=\bigl(\begin{smallmatrix}0&\omega\\-1&0\end{smallmatrix}\bigr)$ without changing the isomorphism type of $P$. Again we define $G:=H\rtimes\langle a_1\rangle$.
Then it turns out that $a_1^2=\bigl(\begin{smallmatrix}\omega&0\\0&\omega\end{smallmatrix}\bigr)\in\Z(\GL(2,q))$. In particular $a_1^2$ is fixed by the element $\alpha\in\Aut_{\mathcal{F}_P(G)}(Q)$ above. So here $\Z(\mathcal{F})=\langle a_1^2\rangle$ is generated by a square in $P$. Thus, we obtain $\mathcal{F}'=\mathcal{F}_P(G)$. 

\textbf{Case~2:} $Q\cong C_{2^t}\times Q_8$ is $\mathcal{F}$-essential in $P$ for some $t\ge 1$.\\
We have seen above that $E$ cannot be dihedral. Hence, $E$ is (generalized) quaternion, i.\,e. $x^2=z$. 
Now $|Q:\Z(Q)|=4$ implies $Q\cap E\cong Q_8$. After conjugation of $Q$ we may assume $Q\cap E=\langle v^{2^{n-2}},x\rangle$. Proposition~\ref{centerfree} implies $z\in\Z(\mathcal{F})$. In particular $Q/\langle z\rangle\cong C_{2^t}\times C_2^2$ is an $\mathcal{F}/\langle z\rangle$-essential subgroup of $P/\langle z\rangle$ (see Theorem~6.3 in \cite{Linckelmann}). So by the first part of the proof and Proposition~\ref{E8normal} (for $n=2$) we get $t=m-1$, and $Q$ is the only $\mathcal{F}$-essential subgroup up to conjugation.
Since $\C_Q(x)\cap\Phi(P)$ is still noncyclic, we have $\C_{\Phi(P)}(x)=\langle\lambda\rangle\times\langle z\rangle\cong C_{2^{m-1}}\times C_2$ as in Case~1. Moreover, $a^2$ fixes $v^{2^{n-2}}$, and it follows that $Q=\langle v^{2^{n-2}},x,\lambda\rangle$.

Here we can handle the uniqueness of $\mathcal{F}$ uniformly without discussing the split and nonsplit case separately. 
Since $\Inn(Q)\cong C_2^2$, $\Aut_{\mathcal{F}}(Q)$ is a group of order $24$ which is generated by $\N_P(Q)/\Z(Q)$ and an automorphism $\alpha\in\Aut_{\mathcal{F}}(Q)$ of order $3$. %(observe however that the isomorphism type of $\N_P(Q)$ depends on $n$). 
Hence, in order to describe the action of $\Aut_{\mathcal{F}}(Q)$ on $Q$ (up to automorphisms from $\Aut(P)$), it suffices to know how $\alpha$ acts on $Q$. 
%Since $a^2v^{2^{n-3}}a^{-2}=v^{2^{n-3}(-1+2^{n-m+1})^2}=v^{2^{n-3}}$, we get $\N_P(Q)=\langle v^{2^{n-3}},x,\lambda\rangle\cong Q_{16}\times C_{2^{m-1}}$.
First of all, $\alpha$ acts on only one subgroup $Q_8\cong R\le Q$. It is not hard to see that $Q'=\langle z\rangle\subseteq R$ and thus $R\unlhd Q$. In particular $R$ is invariant under inner automorphisms of $Q$. Now let $\beta$ be an automorphism of $Q$ coming from $\N_P(Q)/Q\le\Out_{\mathcal{F}}(Q)$. Then $\beta\alpha\equiv\alpha^{-1}\beta\pmod{\Inn(Q)}$. In particular $\beta(R)=\alpha^{-1}(\beta(R))=R$. 
Looking at the action of $\N_P(Q)$, we see that $R\in\{\langle v^{2^{n-2}},x\rangle,\langle v^{2^{n-2}},x\lambda^{2^{m-2}}\rangle\}$. Again the automorphism
\[v\mapsto v,\ x\mapsto x\lambda^{2^{m-2}},\ a\mapsto av^{2^{n-2}}\]
leads to $R=\langle v^{2^{n-2}},x\rangle$. The action of $\alpha$ on $R$ is not quite unique. However, after inverting $\alpha$ if necessary, we have $\alpha(x)\in\{v^{2^{n-2}},v^{-2^{n-2}}\}$. If we conjugate $\alpha$ with the inner automorphism induced by $x$ in doubt, we end up with $\alpha(x)=v^{2^{n-2}}$. Since $\alpha$ has order $3$, it follows that $\alpha(v^{2^{n-2}})=xv^{2^{n-2}}$. So we know precisely how $\alpha$ acts on $R$. 
Since $\alpha$ is unique up to conjugation in $\Aut(Q)$, we have $\C_Q(\alpha)=\Z(Q)=\langle\lambda,z\rangle$. 
Hence, the action of $\Aut_{\mathcal{F}}(Q)$ on $Q$ is uniquely determined. By Alperin's Fusion Theorem, $\mathcal{F}$ is unique up to isomorphism. For the construction of $\mathcal{F}$ we split up the proof again.

\textbf{Case~2(a):} $a^{2^m}=z$.\\
Then again $n>m>1$ and $i=n-m+1$ by Equation~\eqref{congi}. So the isomorphism type of $P$ is determined by $m$ and $n$. 
We construct $P$ and $\mathcal{F}$ in a similar manner as above. For this set $q:=5^{2^{n-2}}$ and $H:=\SL(2,q)$. Then a Sylow $2$-subgroup $H$ is given by $E:=\langle v,x\rangle\cong Q_{2^{n+1}}$ where $v$ and $x$ are defined quite similar as in Case~1(a). The only difference is that $\omega\in\mathbb{F}_q^\times$ has now order $2^n$ and the matrices are not considered modulo $\Z(\SL(2,q))$ anymore. Also the element $a_1$ as above still satisfies $^{a_1}v=v^{-1}$ and $^{a_1}x=vx$. Now we can repeat the calculations in Case~1(a) word by word. Doing so, we obtain $G:=H\langle a\rangle\cong\SL(2,q).C_{2^m}$ and $\mathcal{F}=\mathcal{F}_P(G)$. 

\textbf{Case~2(b):} $a^{2^m}=1$.\\
Here Equation~\eqref{congi} gives $\max(n+m+2,2)\le i\le n$. For every $i$ in this interval we can again construct $P$ and $\mathcal{F}$ in the same manner as before. We omit the details.
%Let $\lambda\in Q$ as above. 
%Then
%\[v^{2^{n-2}}=\lambda^{-1}v^{2^{n-2}}\lambda=a^{-2}v^{2^{n-2}}a^2=v^{2^{n-2}(-1+4i)^2}

\textbf{Case~3:} $Q\cong C_{2^t}\mathop{\ast}Q_8$ is $\mathcal{F}$-essential in $P$ for some $t\ge 2$.\\
Again the argumentation above reveals that $E$ is a quaternion group and $x^2=z$. Moreover, $Q\cap E=\langle v^{2^{n-2}},x\rangle\cong Q_8$ after conjugation if necessary. Going over to $P/\langle z\rangle$, it follows that $t=m$.
Assume $n=m=i$ and $a^{2^m}=z$ for a moment. Then $(ax)^2=vza^2$ and $F_1:=\langle v,ax\rangle\cong C_{2^n}^2$ is maximal in $P$.
%Setting $\widetilde{x}:=ax$ and $\widetilde{y}:=axv$ we obtain
%\begin{align*}
%\widetilde{x}^2&=axax=a^2vz,&\widetilde{y}^2&=\widetilde{x}^2v^2,&[\widetilde{x},\widetilde{y}]&=1.
%\end{align*}
%Hence, $\langle\widetilde{x},\widetilde{y}\rangle\cong C_{2^n}^2$ is maximal in $P$. 
Since $P/\Phi(F_1)$ is nonabelian, we get $P\cong C_{2^n}\wr C_2$ (compare with the proof of Proposition~\ref{rank2}).
%As in the beginning of the proof of Proposition~\ref{rank2} we get $P\cong C_{2^n}\wr C_2$. 
Thus, in case $n=m$ and $a^{2^m}=z$ we assume $i<n$ in the following. We will see later that other parameters cannot lead to a wreath product.
After excluding this special case, it follows as before that $Q$ is the only $\mathcal{F}$-essential subgroup up to conjugation.
Since $\C_Q(x)$ contains an element of order $2^m$, we have $\C_{\Phi(P)}(x)=\langle\lambda\rangle$. 
Hence, we have to replace Equation~\eqref{congi} by
\[z=\lambda^{2^{m-1}}=v^{2^{m+i-2}\nu}a^{2^m}\]
where $\nu$ is an odd number. 
Moreover, $Q=\langle v^{2^{n-2}},x,\lambda\rangle$.
If $a^{2^m}=z$, then $\max(n-m+2,2)\le i\le n$. On the other hand, if $a^{2^m}=1$, then $n>m>1$ and $i=n-m+1$. Hence, these cases complement exactly the Case~2 above. 

The uniqueness of $\mathcal{F}$ is a bit easier than for the other types of essential subgroups. Again $\Aut_{\mathcal{F}}(Q)$ has order $24$ and is generated by $\N_P(Q)/\Z(Q)$ and an automorphism $\alpha\in\Aut_{\mathcal{F}}(Q)$ of order $3$. It suffices to describe the action of $\alpha$ on $Q$ up to automorphisms from $\Aut(P)$. By considering $Q/Q'\cong C_{2^{m-1}}\times C_2^2$ we see that $R:=\langle v^{2^{n-2}},x\rangle$ is the only subgroup of $Q$ isomorphic to $Q_8$. In particular $\alpha$ must act on $R$. Here we also can describe the action precisely by changing $\alpha$ slightly. Moreover, $\C_Q(\alpha)=\Z(Q)=\langle\lambda\rangle$, since $\alpha$ is unique up to conjugation in $\Aut(Q)$. This shows that $\mathcal{F}$ is uniquely determined up to isomorphism. Now we distinguish the split and nonsplit case in order to construct $P$ and $\mathcal{F}$.

\textbf{Case~3(a):} $a^{2^m}=1$.\\
At first glance one might think that the construction in Case~2 should not work here. However, it does. We denote $q$, $H$ and so on as in Case~2(a). Then $a^{2^m}$ is the inner automorphism on $H$ induced by $z$. But since $z\in\Z(H)$, $a^{2^m}$ is in fact the trivial automorphism. Hence, we can construct the semidirect product $G=H\rtimes\langle a\rangle$ which does the job.

\textbf{Case~3(b):} $a^{2^m}=z$.\\
Here we do the opposite as in Case~3(a). With the notation of Case~3(a), $a$ is an automorphism of $H$ such that $a^{2^m}=1$ and $a$ fixes $z\in\Z(H)$. Using Theorem~15.3.1 in \cite{Hall} 
%Beispiel~I.14.8 in \cite{Huppert} 
we can build a nonsplit extension $G:=H\langle a\rangle$ such that $a^{2^m}=z$. This group fulfills our conditions.

Finally we show that different parameters in all these group presentations give nonisomorphic groups. 
Obviously the metacyclic groups are pairwise nonisomorphic and not isomorphic to nonmetacyclic groups. Hence, it suffices to look at the groups coming from Theorem~4.4 in \cite{Janko}. %We have already seen that specific parameters lead 
So let $P$ be as in Equation~\eqref{pres} together with additional dependence between $x^2$ and the choice of $i$ as in the statement of our theorem (this restriction is important).
%one these groups given in the last six cases of the statement of the theorem. Then $P$ has parameters $n$, $m$, $x^2,a^{2^m}\in\{1,z\}$ and $i\in\{\max(n-m+1,2),\ldots,n\}$. 
Assume that $P$ is isomorphic to a similar group $P_1$ where we attach an index $1$ to all elements and parameters of $P_1$. Then we have $2^{n+m+1}=|P|=|P_1|=2^{n_1+m_1+1}$ and $2^n=|P'|=|P_1'|=2^{n_1}$. This already shows $n=n_1$ and $m=m_1$. 
As proved above $P$ admits a nonnilpotent fusion system with essential subgroup $C_{2^{m-1}}\times C_2^2$ if and only if $x^2=1$. Hence, $x^2=1$ if and only if $x_1^2=1$. Now we show $i=i_1$. For this we consider $\Phi(P)=\langle v,a^2\rangle$. Since $\Phi(P)$ is metacyclic, it follows that $\Phi(P)'=\langle[v,a^2]\rangle=\langle v^{2^{i+1}}\rangle\cong C_{2^{\eta}}$ where $\eta:=\max(n-i-1,0)$. Since $i,i_1\le n$, we may assume $i,i_1\in\{n-1,n\}$. In case $i=n$ the subgroup $C:=\langle v,ax\rangle$ is abelian. 
By Theorem~4.3(f) in \cite{Janko}, $C$ is a metacyclic maximal subgroup of $P$.
However, in case $i=n-1$ it is easy to see that the two metacyclic maximal subgroups $\langle v,a\rangle$ and $\langle v,ax\rangle$ of $P$ are both nonabelian. This gives $i=i_1$. 
It remains to show: $a^{2^m}=1\Longleftrightarrow a_1^{2^{m_1}}=1$. For this we may assume $x^2=z$ and $x_1^2=z_1$. 
In case $i=n-m+1$ (and $n>m>1$) we have $a^{2^m}=1$ if and only if $P$ provides a fusion system with essential subgroup $C_{2^m}\mathop{\ast}Q_8$. A similar equivalence holds for $\max(n-m+2,2)\le i\le n$ (even in case $n=m=i$). This completes the proof. 
%Then every elements in $E=\langle v,x\rangle$ of order $4$ are conjugate to $x$ in $P$. 
%Now the restrictions on $i$ and $i_1$ also imply: $a^{2^m}=1\Longleftrightarrow a_1^{2^{m_1}}=1$.
\end{proof}

We present an example to shed more light on the alternative in part~\eqref{dsp} of Theorem~\ref{main}. Let us consider the smallest case $n=m=i=2$. The group $N:=A_6\cong\PSL(2,3^2)$ has Sylow $2$-subgroup $D_8$.
%$S:=\langle(1,2,3,4)(5,6),(2,4)(5,6)\rangle\cong D_8$. 
Let $H:=\langle h\rangle\cong C_4$. It is well known that $\Aut(N)/N\cong C_2^2$, and the three subgroups of $\Aut(N)$ of index $2$ are isomorphic to $S_6$, $\PGL(2,9)$ and the Mathieu group $M_{10}$ of degree $10$ respectively. We choose two homomorphisms $\phi_j:H\to\Aut(N)$ for $j=1,2$ such that $\phi_1(h)\in\PGL(2,9)\setminus N$ is an involution and $\phi_2(h)\in M_{10}\setminus N$ has order $4$ (we do not define $\phi_j$ precisely). Then it turns out that the groups $G_j:=N\rtimes_{\phi_j} H$ for $j=1,2$ have Sylow $2$-subgroup $P$ as in part~\eqref{dsp}. Moreover, one can show that $\mathcal{F}_1:=\mathcal{F}_P(G_1)\ne\mathcal{F}_P(G_2)=:\mathcal{F}_2$. More precisely, $\Z(\mathcal{F}_1)=\Z(G_1)=\langle\phi_1(h)^2\rangle$ is generated by a square in $P$ and $\Z(\mathcal{F}_2)$ is not. 
The indices of $G_j$ in the “Small Group Library” are \texttt{[1440,4592]} and \texttt{[1440,4595]} respectively. It should be clarified that this phenomenon is not connected to the special behavior of $A_6$, since it occurs for all $n$ with $\PSL(2,5^{2^{n-1}})$.

As a second remark we indicate a more abstract way to establish the nonexoticness of our fusion systems. It suffices to look at the cases~\eqref{dnonsp} to \eqref{last} in Theorem~\ref{main}. If $P$ does not contain an abelian $\mathcal{F}$-essential subgroup, then Proposition~\ref{centerfree} shows $\Z(\mathcal{F})\ne 1$. Here Theorem~2.4(b) in \cite{OliverVentura} reduces the question of exoticness to a fusion system on the smaller bicyclic group $P/\langle z\rangle$. Hence, we may assume that there is a $\mathcal{F}$-essential subgroup $Q=\langle z,x,\lambda\rangle\cong C_{2^{m-1}}\times C_2^2$. Moreover, we can assume that $\Z(\mathcal{F})=1$. Now we construct the \emph{reduced} fusion system of $\mathcal{F}$ (see Definition~2.1 in \cite{AOV}). By Proposition~1.5 in \cite{AOV} we have $\pcore_2(\mathcal{F})\le Q\cap {^aQ}\subseteq\langle z,\lambda\rangle$. Since $\pcore_2(\mathcal{F})$ is strongly closed in $P$, we have $z\notin\pcore_2(\mathcal{F})$. Hence, $\pcore_2(\mathcal{F})$ is cyclic and $\Omega(\pcore_2(\mathcal{F}))\subseteq\Z(\mathcal{F})=1$. This shows $\pcore_2(\mathcal{F})=1$. So in the definition of the reduced fusion system we have $\mathcal{F}_0=\mathcal{F}$. Now we determine $\mathcal{F}_1:=\pcore^2(\mathcal{F})$. Since $E=\langle x,vx\rangle$, it turns out that the hyperfocal subgroup of $\mathcal{F}$ is $E\cong D_{2^{n+1}}$. Using Definition~1.21 and 1.23 in \cite{AOV} it is easy to see that $\mathcal{F}_1$ has two essential subgroups isomorphic to $C_2^2$ up to conjugation. That is $\mathcal{F}_1=\mathcal{F}_E(\PSL(2,5^{2^{n-1}}))$. Moreover, we have $\mathcal{F}_2:=\pcore^{2'}(\mathcal{F}_1)=\mathcal{F}_1$. So it follows that $\mathcal{F}_1$ is the reduction of $\mathcal{F}$. 
By Proposition~4.3 in \cite{AOV}, $\mathcal{F}_1$ is tame in the sense of Definition~2.5 in \cite{AOV}. Without using the classification of the finite simple groups, Theorem~2.10 in \cite{AOV} implies that $\mathcal{F}_1$ is even strongly tame.
%%As announced in the talk \cite{Brototalk} every fusion system on a finite simple group of Lie type is tame in the sense of Definition~2.5 in \cite{AOV}. So $\mathcal{F}_1$ is tame. 
Hence, also $\mathcal{F}$ is tame by Theorem~2.20 in \cite{AOV}. In particular $\mathcal{F}$ is not exotic. 
However, using this approach it is a priori not clear if these (nonnilpotent) fusion systems exist at all. 
%But this can also be handled in an abstract manner as follows. 
%Let $Q$ be (a candidate for) an $\mathcal{F}$-essential subgroup of $P$. 
%Using Proposition~\ref{essnotnormal}, it is easy to see there is a nonnilpotent fusion system on $N:=\N_P(Q)$ which can be realized by a finite group $H$ (see also Theorem~4.6 in \cite{Linckelmann}). Then Theorem~1 in \cite{RobinsonAmal} shows that $\mathcal{F}$ is the fusion system of the (infinite!) free product $H\mathop{\ast}_N P$ with amalgamated subgroup $N$. The only problem which remains on these lines is the uniqueness of $\mathcal{F}$. The different possibilities for $\mathcal{F}$ differ by the ways one can embed $N$ into $H$ in the construction of $H\mathop{\ast}_N P$.

As another comment, we observe that the $2$-groups in parts~\eqref{q8} to \eqref{last} have $2$-rank $2$. Hence, these are new examples in the classification of all fusion systems on $2$-groups of $2$-rank $2$ which was started in \cite{CravenGlesser}. It is natural to ask what happens if we interchange the restrictions on $i$ in case~\eqref{dnonsp} and case~\eqref{dsp} in Theorem~\ref{main}. We will see in the next theorem that this does not result in new groups.

\begin{Theorem}
Let $P$ be a bicyclic, nonmetacyclic $2$-group. Then $P$ admits a nonnilpotent fusion system if and only if $P'$ is cyclic.
\end{Theorem}
\begin{proof}
By Theorem~\ref{cycliccom} it suffices to prove only one direction. Let us assume that $P'$ is cyclic. Since $P$ is nonmetacyclic, it follows that $P'\ne 1$. In case $|P'|=2$, Theorem~4.1 in \cite{Janko} implies that $P$ is minimal nonabelian of type $(n,1)$ for some $n\ge 2$. We have already shown that there is a nonnilpotent fusion system on this group. Thus, we may assume $|P'|>2$. Then we are again in Theorem~4.4 in \cite{Janko}. After adapting notation, $P$ is given as in Equation~\eqref{pres}. In case $x^2=z$ there is always a nonnilpotent fusion system on $P$ by Theorem~\ref{main}. Hence, let $x^2=1$. Then it remains to deal with two different pairs of parameters.

\textbf{Case~1:} $a^{2^m}=1$ and $i=n-m+1\ge 2$.\\
Set $\widetilde{x}:=xa^{2^{m-1}}$. Then 
\[\widetilde{x}^2=xa^{2^{m-1}}xa^{2^{m-1}}=(v^{-1}a)^{2^{m-1}}a^{2^{m-1}}=v^{2^{i+m-2}\nu}a^{2^m}=z\]
for an odd integer $\nu$. Moreover, $^{\widetilde{x}}v=v^{-1}$, $^a\widetilde{x}=vxa^{2^{m-1}}=v\widetilde{x}$. This shows that $P$ is isomorphic to a group with parameters $x^2=z$, $a^{2^m}=1$ and $i=n-m+1\ge 2$. In particular Theorem~\ref{main} provides a nonnilpotent fusion system on $P$.

\textbf{Case~2:} $a^{2^m}=z$ and $\max(2,n-m+2)\le i\le n$.\\
Again let $\widetilde{x}:=xa^{2^{m-1}}$. Then
\[\widetilde{x}^2=v^{2^{i+m-2}\nu}a^{2^m}=z.\]
Hence, $P$ is isomorphic to a group with parameters $x^2=a^{2^m}=z$ and $\max(2,n-m+2)\le i\le n$. The claim follows as before.
\end{proof}

Now we count how many interesting fusion systems we have found.

\begin{Proposition}
Let $f(N)$ be the number of isomorphism classes of bicyclic $2$-groups of order $2^N$ which admit a nonnilpotent fusion system. Moreover, let $g(N)$ be the number of nonnilpotent fusion systems on all bicyclic $2$-groups of order $2^N$. Then
\begin{center}
\begin{tabular}{c|c|c|c|c|c}
$N$&$1$&$2$&$3$&$\ge 4$ even&$\ge 5$ odd\\\hline
$f(N)$&$0$&$1$&$2$&$\frac{3}{4}N^2-3N+5$&\vphantom{$5^{4^4}$}$(3N^2+1)/4-3N+3$\\
$g(N)$&$0$&$1$&$3$&$\frac{3}{4}N^2-2N+5$&\vphantom{$5^{4^4}$}$(3N^2+1)/4-2N+5$
\end{tabular}
\end{center}
\end{Proposition}
\begin{proof}
Without loss of generality, $N\ge 4$.
We have to distinguish between the cases $N$ even and $N$ odd. Assume first that $N$ is even. Then we get the following five groups: $C_{2^{N/2}}^2$, $D_{2^N}$, $Q_{2^N}$, $SD_{2^N}$ and the minimal nonabelian group of type $(N-2,1)$. From case~\eqref{dnonsp} of Theorem~\ref{main} we obtain exactly $N/2-2$ groups. In case~\eqref{dsp} the number of groups is
\begin{align*}
\sum_{n=2}^{N-3}{(n-\max(2,2n-N+3)+1)}&=\sum_{n=2}^{N/2-1}{(n-1)}+\sum_{n=N/2}^{N-3}{(N-n-2)}=2\sum_{n=1}^{N/2-2}{n}\\
&=(N/2-2)(N/2-1)=\frac{N^2}{4}-\frac{3N}{2}+2.
\end{align*}
The other cases are similar (observe that the wreath product cannot occur, since $N$ is even). All together we get
\[5+3(N/2-2)+3(N^2/4-3N/2+2)=\frac{3}{4}N^2-3N+5\]
bicyclic $2$-groups of order $2^N$ with nonnilpotent fusion system. 

Now if $N$ is odd we have the following four examples: $D_{2^N}$, $Q_{2^N}$, $SD_{2^N}$ and the minimal nonabelian group of type $(N-2,1)$. From case~\eqref{dnonsp} of Theorem~\ref{main} we obtain exactly $(N-5)/2$ groups. In case~\eqref{dsp} the number of groups is
\begin{align*}
\sum_{n=2}^{N-3}{(n-\max(2,2n-N+3)+1)}&=\sum_{n=2}^{(N-1)/2}{(n-1)}+\sum_{n=(N+1)/2}^{N-3}{(N-n-2)}=2\sum_{n=1}^{(N-5)/2}{n}+(N-3)/2\\
&=\frac{(N-5)(N-3)}{4}+\frac{N-3}{2}=\frac{N^2-6N+9}{4}.
\end{align*}
Adding the numbers from the other cases (this time including the wreath product), we obtain
\[4+3\frac{N^2-4N-1}{4}=\frac{3N^2+1}{4}-3N+3.\]
In order to obtain $g(N)$ from $f(N)$ we have to add one fusion system on $D_{2^N}$, one on $Q_{2^N}$, and two on $SD_{2^N}$. If $N$ is odd, we get two more fusion systems on the wreath product. For all $N\ge 5$ we have to add $N-4$ fusion systems coming from part~\eqref{dsp} in Theorem~\ref{main}.
\end{proof}

\section{Applications}
We present an application to finite simple groups. For this we introduce a general lemma.

\begin{Lemma}\label{perfect}
Let $G$ be a perfect group and $1\ne P\in\Syl_p(G)$ such that $\N_G(P)=P\C_G(P)$. Then there are at least two conjugacy classes of $\mathcal{F}_P(G)$-essential subgroups in $P$. 
\end{Lemma}
\begin{proof}
Let $\mathcal{F}:=\mathcal{F}_P(G)$. If there is no $\mathcal{F}$-essential subgroup, then $\mathcal{F}$ is nilpotent and $G$ is $p$-nilpotent, since $\Out_{\mathcal{F}}(P)=\N_G(P)/P\C_G(P)=1$. Then $G'\le P'\pcore_{p'}(G)<G$, because $P\ne 1$. Contradiction. Now suppose that there is exactly one $\mathcal{F}$-essential subgroup $Q\le P$ up to conjugation. Then $Q$ lies in a maximal subgroup $M<P$. Moreover, $P'\subseteq\Phi(P)\subseteq M$. Now the focal subgroup theorem (see Theorem~7.3.4 in \cite{Gorenstein}) gives the following contradiction:
\[P=P\cap G=P\cap G'=\langle x^{-1}\alpha(x):x\in P,\ \alpha \text{ morphism in }\mathcal{F}\rangle\subseteq P'Q\subseteq M.\qedhere\]
\end{proof}

We remark that the number of conjugacy classes of essential subgroups is sometimes called the \emph{essential rank} of the fusion system (see for example \cite{Henke}). 

\begin{Theorem}
Let $G$ be a simple group with bicyclic Sylow $2$-subgroup. Then $G$ is one of the following groups: $C_2$, $\PSL(i,q)$, $\PSU(3,q)$, $A_7$ or $M_{11}$ for $i\in\{2,3\}$ and $q$ odd.
\end{Theorem}
\begin{proof}
By the Alperin-Brauer-Gorenstein Theorem \cite{2rank2} on simple groups of $2$-rank $2$, we may assume that $G$ has $2$-rank $3$ (observe that a Sylow $2$-subgroup of $\PSU(3,4)$ is not bicyclic, since it has rank $4$). 
%Now we could apply the Gorenstein-Harada result on simple groups of sectional rank at most $4$. However, we prefer to give a more elementary argument. 
Let $P\in\Syl_2(G)$ and $\mathcal{F}:=\mathcal{F}_P(G)$. By Theorem~\ref{main}, there is only one $\mathcal{F}$-essential subgroup $Q$ in $P$ up to conjugation. But this contradicts Lemma~\ref{perfect}.
\end{proof}

Now we consider fusion systems coming from block theory. 
Let $B$ be a $p$-block of a finite group $G$. We denote the number of irreducible ordinary characters of $B$ by $k(B)$ and the number of irreducible Brauer characters of $B$ by $l(B)$. Moreover, let $k_0(B)$ be the number of irreducible characters of \emph{height} $0$, i.\,e. the $p$-part of the degree of these characters is as small as possible.
Let $D$ be a defect group of $B$. Then for every element $u\in D$ we have a \emph{subsection} $(u,b_u)$ where $b_u$ is a Brauer correspondent of $B$ in $\C_G(u)$. 

\begin{Theorem}
Olsson's Conjecture holds for all blocks of finite groups with bicyclic defect groups.
\end{Theorem}
\begin{proof}
Let $B$ be a $p$-block of a finite group with bicyclic defect group $D$. Since all bicyclic $p$-groups for an odd prime are metacyclic, we may assume $p=2$.
If $D$ is metacyclic, minimal nonabelian or a wreath product, then Olsson's Conjecture holds by the results in \cite{Sambale,Sambalemna,Kuelshammerwr}. 
Let $\mathcal{F}$ be the fusion system of $B$. Without loss of generality, $\mathcal{F}$ is nonnilpotent.
Hence, we may assume that $D$ is given by
\[D\cong\langle v,x,a\mid v^{2^n}=1,\ x^2,a^{2^m}\in\langle v^{2^{n-1}}\rangle,\ ^xv=v^{-1},\ ^av=v^{-1+2^i},\ ^ax=vx\rangle\]
where $\max(2,n-m+1)\le i\le n$ as in Theorem~\ref{main}. Moreover, there is only one conjugacy classes of $\mathcal{F}$-essential subgroups of $D$. 
%In order to prove Brauer $k(B)$-Conjecture we apply Theorem~3.4 in \cite{RobinsonNumber}. For this we consider the major subsection $(z,b_z)$ where $z:=v^{2^{n-1}}\in\Z(D)$ as usual. Then $b_z$ has also defect group $D$ and $z$ lies in the center of the fusion system of $b_z$. However, we have seen in the proof of Theorem~\ref{main} that there is always an automorphism of an $\mathcal{F}$-essential subgroup which does not fix $z$. Hence, the block $b_z$ is nilpotent, and Theorem~3.4 in \cite{Robinson} implies $k(B)\le|D|$, i.\,e. Brauer's $k(B)$-Conjecture is fulfilled.
We use Proposition~2.5(i) in \cite{HKS}. For this let us consider the subsection $(a,b_a)$. Since $\langle a,v\rangle$ is a metacyclic maximal subgroup of $P$, $a$ does not lie in any $\mathcal{F}$-essential subgroup of $P$. In particular $\langle a\rangle$ is fully $\mathcal{F}$-centralized. Thus, Theorem~2.4(ii) in \cite{Linckelmann2} implies that $b_a$ has defect group $\C_D(a)$. Obviously, $\C_{\langle v\rangle}(a)=\langle z\rangle$. Now let $v^jx\in\C_D(a)$ for some $j\in\mathbb{Z}$. Then $v^jx={^a(v^jx)}=v^{1-j+2^ij}x$ and $v^{2j}=v^{1+2^ij}$, a contradiction. This shows $\C_D(a)=\langle a,z\rangle$. Now by Proposition~2.5(i) in \cite{HKS} we obtain $k_0(B)\le\lvert\C_D(a)\rvert=2^{m+1}=|D:D'|$, i.\,e. Olsson's Conjecture holds.
\end{proof}

Using Theorem~3.4 in \cite{RobinsonNumber}, it is not hard to see that also Brauer's $k(B)$-Conjecture holds if for the fusion system of $B$ one of the cases \eqref{nilcase} to \eqref{dsp} in Theorem~\ref{main} occurs. 
%We plan to investigate this and other conjectures on blocks with bicyclic defect groups in a separate paper.

\section*{Acknowledgment}
This work was supported by the German Academic Exchange Service (DAAD) and the Carl Zeiss Foundation. It was written mostly in Santa Cruz, USA. I thank the University of California for its hospitality.

\begin{center}
Benjamin Sambale\\
Institut für Mathematik\\
Friedrich-Schiller-Universität\\
07743 Jena\\
Germany\\
\href{mailto:benjamin.sambale@uni-jena.de}{\texttt{benjamin.sambale@uni-jena.de}}
\end{center}

\end{document}